\documentclass[12pt]{article}

\usepackage{amsmath, amsthm, amssymb, latexsym, enumerate}

\usepackage[all]{xy}

\newcommand{\sub}[1]{_{\textup{{#1}}}}
\renewcommand{\sup}[1]{^{\textup{{#1}}}}

\renewcommand{\phi}{\varphi}

\renewcommand{\geq}{\geqslant}
\renewcommand{\leq}{\leqslant}

\renewcommand{\nleq}{\nleqslant}

\newcommand{\RCA}{\mathsf{RCA}}

\newtheorem{thm}{Theorem}[section]
\newtheorem{lem}[thm]{Lemma}
\newtheorem{cor}[thm]{Corollary}
\newtheorem{prop}[thm]{Proposition}

\theoremstyle{definition}
\newtheorem{defn}[thm]{Definition}
\newtheorem{rem}[thm]{Remark}
\newtheorem{exa}[thm]{Example}
\newtheorem{q}[thm]{Question}

\numberwithin{equation}{section}

\DeclareMathOperator{\dom}{dom}

\renewcommand{\footnotemark}{}

\title{Reduction Games, Provability, and Compactness}

\author{Damir D. Dzhafarov, Denis R. Hirschfeldt,\\ and Sarah
C. Reitzes\thanks{The authors were partially supported by a Focused
Research Group grant from the National Science Foundation of the
United States, DMS-1854355 (Connecticut) and DMS-1854279
(Chicago). Hirschfeldt was also partially support by NSF grant
DMS-1600543. Reitzes was also partially supported by DGE-1746045.  The authors thank Jeff Hirst, Alberto Marcone, Carl Mummert,
Ludovic Patey, Arno Pauly, Richard Shore, Patrick Uftring, Keita Yokoyama, as well as the anonymous referee for valuable comments and suggestions during the writing
of this paper. They thank the Casa Matem\'atica Oaxaca for hosting a workshop
during which useful conversations regarding this paper occurred.}}

\begin{document}

\maketitle

\begin{abstract}
Hirschfeldt and Jockusch (2016) introduced a two-player game in which
winning strategies for one or the other player precisely correspond to
implications and non-implications between $\Pi^1_2$ principles over
$\omega$-models of $\RCA_0$. They also introduced a version of this
game that similarly captures provability over $\RCA_0$. We generalize
and extend this game-theoretic framework to other formal systems, and
establish a certain compactness result that shows that if an
implication $\mathsf{Q} \to \mathsf{P}$ between two principles holds,
then there exists a winning strategy that achieves victory in a number
of moves bounded by a number independent of the specific run of the
game. This compactness result generalizes an old proof-theoretic fact
noted by H.~Wang (1981), and has applications to the reverse
mathematics of combinatorial principles.

We also demonstrate how this framework leads to a new kind of analysis
of the logical strength of mathematical problems that refines both
that of reverse mathematics and that of computability-theoretic
notions such as Weihrauch reducibility, allowing for a
kind of fine-structural comparison between $\Pi^1_2$ principles that
has both computability-theoretic and proof-theoretic aspects, and can
help us distinguish between these, for example by showing that a
certain use of a principle in a proof is ``purely proof-theoretic'',
as opposed to relying on its computability-theoretic strength.

We give examples of this analysis to a number of principles at the
level of $\mathsf{B}\Sigma^0_2$, uncovering new differences between
their logical strengths.
\end{abstract}

\section{Introduction}

Reverse mathematics gives us a way to compare the relative strength of
theorems by establishing implications and nonimplications over a weak
subsystem of second-order arithmetic, typically $\mathsf{RCA}_0$,
which corresponds roughly to computable mathematics. (We will assume
some familiarity with reverse mathematics and computability
theory. Standard resources in these areas include~\cite{Simpson}
and~\cite{Soare}, respectively.) In many cases, nonimplications over
$\mathsf{RCA}_0$ are proved using $\omega$-models, i.e., models of
$\mathsf{RCA}_0$ with standard first-order part. We say that
$\mathsf{P}$ is \emph{$\omega$-reducible} to $\mathsf{Q}$, and write
$\mathsf{P} \leq\sub{$\omega$} \mathsf{Q}$, if every $\omega$-model of
$\mathsf{RCA}_0 + \mathsf{Q}$ is a model of $\mathsf{P}$. 

Implication over $\mathsf{RCA}_0$ and $\omega$-reducibility are not
fine enough for some purposes, so other notions of
computability-theoretic reduction between theorems have been
extensively studied. These are particularly well-adapted to the
following class of theorems, which includes a large proportion of
those that have been studied in reverse mathematics: A
\emph{$\Pi^1_2$-problem} is a sentence $\forall X\,[\Theta(X) \,
\rightarrow \, \exists Y\, \Psi(X,Y)]$ of second-order arithmetic
such that $\Theta$ and $\Psi$ are arithmetic. The term ``problem''
reflects a computability-theoretic view that sees such a sentence as a
process of finding a suitable $Y$ given $X$. In line with this view,
we say that an \emph{instance} of this problem is an $X \subseteq
\omega$ such that $\Theta(X)$ holds, and a \emph{solution} to this
problem is a $Y \subseteq \omega$ such that $\Psi(X,Y)$ holds.

For example, the following versions of Ramsey's Theorem are
$\Pi^1_2$-problems that have been extensively studied in reverse
mathematics and computability theory, and will be useful sources of
examples for us as well. (We often state $\Pi^1_2$-problems in ways
that make mention of objects other than natural numbers and sets of
natural numbers. We assume these are coded in an appropriate way. For
combinatorial objects like the ones below, these codings are
straightforward and do not affect the analysis of these problems.)

\begin{defn}
\label{rtdef}
For a set $X$, let $[X]^n$ be the collection of $n$-element subsets of
$X$. A \emph{$k$-coloring} of $[X]^n$ is a map $c : [X]^n \rightarrow
k$. A coloring of $[X]^2$ is \emph{stable} if $\lim_{y \in X} c(x,y)$
exists for all $x \in X$. A set $H \subseteq X$ is \emph{homogeneous}
for $c : [X]^n \rightarrow k$ if there is an $i$ such that $c(s)=i$
for all $s \in [H]^n$. A set $L \subseteq X$ is
\emph{limit-homogeneous} for $c : [X]^2 \rightarrow k$ if there is an
$i$ such that $\lim_{y \in L} c(x,y)=i$ for all $x \in L$.
\begin{enumerate}

\item $\mathsf{RT}^n_k$: Every $k$-coloring of $[\mathbb N]^n$ has an
infinite homogeneous set.

\item $\mathsf{RT}^n_{<\infty}$: $\forall k\, \mathsf{RT}^n_k$.

\item $\mathsf{RT}$: $\forall n\, \forall k\, \mathsf{RT}^n_k$.

\item $\mathsf{SRT}^2_k$: Every stable $k$-coloring of $[\mathbb N]^2$
has an infinite homogeneous set.

\item $\mathsf{D}^2_k$: Every stable $k$-coloring of $[\mathbb N]^2$
has an infinite limit-homogeneous set.

\end{enumerate}
\end{defn}

It is well-known that $\mathsf{RT}^n_k$ and $\mathsf{RT}^n_{<\infty}$
are equivalent to $\mathsf{ACA}_0$ for each $n \geq 3$ (we always
assume $k \geq 2$), while $\mathsf{RT}^1_k$ is provable in
$\mathsf{RCA}_0$. (We will discuss $\mathsf{RT}^1_{<\infty}$ and
$\mathsf{RT}$ below. For more on the computability theory and reverse
mathematics of these principles, see~\cite{Hirschfeldt}.)  The
question of whether $\mathsf{SRT}^2_2$ implies $\mathsf{RT}^2_2$
motivated a great deal of research since being raised by Cholak,
Jockusch, and Slaman~\cite{CJS}. Chong, Slaman, and Yang~\cite{CSY}
showed that $\mathsf{RCA}_0 \nvdash \mathsf{SRT}^2_2 \rightarrow
\mathsf{RT}^2_2$, with a proof that made essential use of
non-$\omega$-models. Recently, Monin and Patey~\cite{MonPat} have
finally shown that $\mathsf{RT}^2_2 \nleq\sub{$\omega$}
\mathsf{SRT}^2_2$. The relationship between $\mathsf{SRT}^2_2$ and
$\mathsf{D}^2_2$ is also interesting, and will be discussed in
Section~\ref{exsec}.

Hirschfeldt and Jockusch~\cite{HirJoc} gave characterizations of both
$\mathsf{P} \leq\sub{$\omega$} \mathsf{Q}$ and $\mathsf{RCA}_0 \vdash
\mathsf{Q} \rightarrow \mathsf{P}$ for $\Pi^1_2$-problems $\mathsf{P}$
and $\mathsf{Q}$ in terms of winning strategies in certain games. In
this paper, we study further aspects of the latter characterization
and generalizations of it, in particular establishing a compactness
theorem that shows that certain winning strategies can always be
chosen to win in a number of moves bounded by a number independent of
the instance of $\mathsf{P}$ being considered. As explained below,
this theorem can be seen as a generalization of a metatheorem about
$\mathsf{ACA}_0$. This metatheorem has been used, for instance, to translate
computability-theoretic results of Jockusch~\cite{Jockusch} into a
proof that $\mathsf{ACA}_0 \nvdash \mathsf{RT}$.

The difference between the two game-theoretic characterizations
in~\cite{HirJoc} is that for $\omega$-reducibility, the games are
played over the standard natural numbers, while for provability over
$\mathsf{RCA}_0$ they are played over possibly nonstandard models of
$\Sigma^0_1\textup{-}\mathsf{PA}$ (the first-order part of
$\mathsf{RCA}_0$). We hope to show in this paper that there is a rich
theory to be obtained by generalizing computability-theoretic
reductions between $\Pi^1_2$-problems to models of subsystems of
second-order arithmetic with possibly nonstandard first-order parts,
and to begin its systematic development. In particular, this theory
allows us to conduct a fine-structural comparison between such
problems that has both computability-theoretic and proof-theoretic
aspects, and can help us distinguish between these, for example by
showing that a certain use of a principle in a proof is ``purely
proof-theoretic'', as opposed to relying on its
computability-theoretic strength.

\bigskip

Computable reducibility and Weihrauch reducibility are two of the most
widely-studied notions of computability-theoretic reducibility between
$\Pi^1_2$-problems. The latter (in a more general form) has a long
history, particularly in computable analysis (see e.g.\ \cite{BGP}),
while the former was introduced by Dzhafarov~\cite{Dzhafarov}.

\begin{defn}
Let $\mathsf{P}$ and $\mathsf{Q}$ be $\Pi^1_2$-problems.

We say that $\mathsf{P}$ is \emph{computably reducible} to
$\mathsf{Q}$, and write $\mathsf{P} \leq\sub{c} \mathsf{Q}$, if for
every instance $X$ of $\mathsf{P}$, there is an $X$-computable
instance $\widehat{X}$ of $\mathsf{Q}$ such that, for every solution
$\widehat{Y}$ to $\widehat{X}$, there is an $X \oplus
\widehat{Y}$-computable solution to $X$.

We say that $\mathsf{P}$ is \emph{Weihrauch reducible} to
$\mathsf{Q}$, and write $\mathsf{P} \leq\sub{W} \mathsf{Q}$, if there
are Turing functionals $\Phi$ and $\Psi$ such that, for every instance
$X$ of $\mathsf{P}$, the set $\widehat{X}=\Phi^X$ is an instance of
$\mathsf{Q}$, and for every solution $\widehat{Y}$ to $\widehat{X}$,
the set $Y=\Psi^{X \oplus \widehat{Y}}$ is a solution to $X$.
\end{defn}

These two reducibilities allow us to use only a single instance of
$\mathsf{Q}$ in solving an instance of $\mathsf{P}$. To generalize
these notions to allow multiple instances of $\mathsf{Q}$ to be used,
Hirschfeldt and Jockusch~\cite{HirJoc} defined the following game.

\begin{defn}
\label{redgame}
Let $\mathsf{P}$ and $\mathsf{Q}$ be $\Pi^1_2$-problems. The
\emph{reduction game} $G(\mathsf{Q} \rightarrow \mathsf{P})$ is a
two-player game played according to the following rules.
\begin{enumerate}[\rm (1)]

\item If at any point a player cannot make a move, the opponent wins.

\item If one of the players wins, the game ends.

\item On the first move, Player 1 plays an instance $X_0$ of
$\mathsf{P}$. Then Player 2 either plays an $X_0$-computable solution
to $X_0$ and wins, or plays an $X_0$-computable instance
$Y_1$ of $\mathsf{Q}$.

\item For $n > 1$, on the $n$th move, Player 1 plays a solution
$X_{n-1}$ to the instance $Y_{n-1}$ of $\mathsf{Q}$. Then Player 2
either plays an $(X_0 \oplus \cdots \oplus X_{n-1})$-computable
solution to $X_0$ and wins, or plays an $(X_0 \oplus \cdots \oplus
X_{n-1})$-computable instance $Y_n$ of $\mathsf{Q}$.

\item If the game never ends then Player 1 wins.

\end{enumerate}
\end{defn}

A winning strategy for Player 2 in this game is a form of generalized
computable reduction. Hirschfeldt and Jockusch~\cite{HirJoc} showed
that if $\mathsf{P} \leq\sub{$\omega$} \mathsf{Q}$ then Player 2 has a
winning strategy for $G(\mathsf{Q} \rightarrow \mathsf{P})$, while
otherwise Player 1 has a winning strategy for $G(\mathsf{Q}
\rightarrow \mathsf{P})$, so generalized computable reducibility is
actually the same as $\omega$-reducibility. They then defined an
analogous notion of generalized Weihrauch reducibility, where
$\mathsf{P} \leq\sub{gW} \mathsf{Q}$ if Player 2 has a uniformly
computable winning strategy for $G(\mathsf{Q} \rightarrow
\mathsf{P})$. (See~\cite{HirJoc} for the details of this definition.)
Neumann and Pauly~\cite{NeuPau} gave an equivalent definition in terms
of an operator ${}^\diamond$ on the Weihrauch degrees. (See also \cite{Westrick}
for some more recent discussion of, and results about, the ${}^\diamond$
operator.)

\bigskip

We can generalize the notions of instance and solution of a
$\Pi^1_2$-problem $\mathsf{P} \equiv \forall X\,[\Theta(X) \,
\rightarrow \, \exists Y\, \Psi(X,Y)]$ to possibly nonstandard
structures in the language of first-order arithmetic in a natural
way. We denote the languages of first- and second-order arithmetic by
$L_1$ and $L_2$, respectively. Let $M$ be an $L_1$-structure. We
denote the domain of $M$ by $|M|$. For $S \subseteq |M|$, we denote
the $L_2$-structure with first-order part $M$ and second-order part
$S$ by $(M,S)$. For an $L_1$-structure $M$, an \emph{$M$-instance of
$\mathsf P$} is an $X \subseteq |M|$ such that $(M,\{X\}) \vDash
\Theta(X)$, and a \emph{solution} to this instance is a $Y \subseteq
|M|$ such that $(M,\{X,Y\}) \vDash \Psi(X,Y)$.

Hirschfeldt and Jockusch~\cite[Section 4.5]{HirJoc} noted that
reduction games can be extended to possibly nonstandard countable
models of $\Sigma^0_1\textup{-}\mathsf{PA}$ (i.e., first-order parts
of models of $\mathsf{RCA}_0$), with $\Delta^0_1$-definability playing
the role of computability as follows. For $X_0,\ldots,X_n \subseteq
|M|$, we denote by $M[X_0,\ldots,X_n]$ the $L_2$-structure with
first-order part $M$ and second-order part consisting of all $X \subseteq
|M|$ that are $\Delta^0_1$-definable over $|M| \cup
\{X_0,\ldots,X_n\}$, which means that there are $\Sigma^0_1$ formulas
$\phi_0(x)$ and $\phi_1(x)$ with parameters from $|M| \cup
\{X_0,\ldots,X_n\}$ such that $(M,\{X_0,\ldots,X_n\}) \vDash \forall
x\, (\phi_0(x) \leftrightarrow \neg\phi_1(x))$ and $X = \{n \in
|\mathcal M| : (M,\{X_0,\ldots,X_n\}) \vDash \phi_0(n)\}$.

\begin{defn}
\label{generalgames}
Let $\mathsf{P}$ and $\mathsf{Q}$ be $\Pi^1_2$-problems. The
\emph{$\mathsf{RCA}_0$-reduction game} $G^{\mathsf{RCA}_0}(\mathsf{Q}
\rightarrow \mathsf{P})$ is a two-player game played according to the
following rules.
\begin{enumerate}[\rm (1)]

\item If at any point a player cannot make a move, the opponent wins.

\item If one of the players wins, the game ends.

\item On the first move, Player 1 plays a countable $L_1$-structure
$M$ and an $M$-instance $X_0$ of $\mathsf{P}$ such that $M[X_0]
\vDash \mathsf{RCA}_0$. Then Player 2 either plays a solution to $X_0$
in $M[X_0]$ and wins, or plays an $M$-instance $Y_1$ of $\mathsf{Q}$
in $M[X_0]$.

\item For $n > 1$, on the $n$th move, Player 1 plays a solution
$X_{n-1}$ to the instance $Y_{n-1}$ of $\mathsf{Q}$ such that
$M[X_0,\ldots,X_{n-1}] \vDash \mathsf{RCA}_0$. Then Player 2
either plays a solution to $X_0$ in $M[X_0,\ldots,X_{n-1}]$ and wins,
or plays an $M$-instance $Y_n$ of $\mathsf{Q}$ in
$M[X_0,\ldots,X_{n-1}]$.

\item If the game never ends then Player 1 wins.

\end{enumerate}
\end{defn}

This definition allows us to capture provability over $\mathsf{RCA}_0$
in terms of winning strategies.

\begin{prop}[Hirschfeldt and Jockusch~\cite{HirJoc}]
\label{hjprop}
Let $\mathsf{P}$ and $\mathsf{Q}$ be $\Pi^1_2$-prob\-lems. If
$\mathsf{RCA}_0 \vdash \mathsf{Q} \rightarrow \mathsf{P}$ then Player
2 has a winning strategy for $G^{\mathsf{RCA}_0}(\mathsf{Q}
\rightarrow \mathsf{P})$. Otherwise, Player 1 has a winning strategy
for $G^{\mathsf{RCA}_0}(\mathsf{Q} \rightarrow \mathsf{P})$.
\end{prop}

The proof of this proposition is essentially the same as that of the
analogous result for games over the standard natural numbers and
$\omega$-reducibility in~\cite[Proposition 4.2]{HirJoc}. We will prove
a stronger version in Proposition~\ref{auxprop}.

However, it might be that the above definition is not quite the best
one. In Section~\ref{classicalsec}, we will discuss a modified
game. We will define it for arbitrary subsystems of second-order
arithmetic, but in the case of $\mathsf{RCA}_0$, the modified game
$\widehat{G}^{\mathsf{RCA}_0}(\mathsf{Q} \rightarrow \mathsf{P})$ is
defined as above, except that on its first move, Player 1 must play
not only a countable $L_1$-structure $M$, but a model $\mathcal M$ of
$\mathsf{RCA}_0$ with countable first-order part (but possibly
uncountable second-order part); and from then on, its moves
$X_0,X_1,\ldots$ must all come from $\mathcal M$. This game makes
intuitive sense in that if Player 1 is trying to claim that
$\mathsf{RCA}_0 \nvdash \mathsf{Q} \rightarrow \mathsf{P}$, then it
should be prepared to propose a model of $\mathsf{RCA}_0$ within which
to witness this fact. This idea was not noticed in~\cite{HirJoc}
because in the $\omega$-model case after which the original
$\mathsf{RCA}_0$-reduction game was modeled, there is really no issue,
since Player 1 always automatically plays within a particular model of
$\mathsf{RCA}_0$, namely $(\omega,\mathcal P(\omega))$, where
$\mathcal P(\omega)$ is the full power set of $\omega$. (For a
nonstandard model $M$, of course, the full power set will include a
cut, so we cannot add it to $M$ to obtain a model of RCA$_0$.)

As we will see in Section~\ref{classicalsec}, Proposition~\ref{hjprop}
still holds for this modified game, indeed with the same proof. But we
will also be able to prove a stronger version that shows that a
certain kind of compactness theorem holds in this case: As shown
in~\cite{HirJoc}, for a game $G(\mathsf{Q} \rightarrow \mathsf{P})$
over the standard natural numbers, it is possible that Player 2 has a
winning strategy but there is no $n$ such that Player 2 has a winning
strategy that is guaranteed to win in at most $n$ many moves. As we
will show in Section~\ref{compactnesssec}, for our modified games over
possibly nonstandard models, this will no longer be the case, which
makes sense given that these games capture notions of provability, and
a proof of $\mathsf{Q} \rightarrow \mathsf{P}$ is a finite object.

\begin{thm}
\label{rcathm}
Let $\mathsf{P}$ and $\mathsf{Q}$ be $\Pi^1_2$-problems. If
$\mathsf{RCA}_0 \vdash \mathsf{Q} \rightarrow \mathsf{P}$ then there
is an $n$ such that Player 2 has a winning strategy for
$\widehat{G}^{\mathsf{RCA}_0}(\mathsf{Q} \rightarrow \mathsf{P})$ that
ensures victory in at most $n$ many moves. Otherwise, Player 1 has a
winning strategy for $\widehat{G}^{\mathsf{RCA}_0}(\mathsf{Q}
\rightarrow \mathsf{P})$.
\end{thm}

We do not know whether the first part of this result holds for the
game $G^{\mathsf{RCA}_0}(\mathsf{Q} \rightarrow \mathsf{P})$ as well.

Theorem~\ref{rcathm}, whose proof will in fact use the compactness
theorem for first-order logic, can be seen as a generalization of the
following fact, which appears in Wang~\cite{Wang}, where it is said
that it is ``almost certainly a known theorem in proof theory.'' For a
model-theoretic proof using compactness due to Jockusch,
see~\cite[Section 6.3]{Hirschfeldt}.

\begin{thm}[see Wang~\cite{Wang}]
\label{wangthm}
Let $\mathsf{P} \equiv \forall X\,[\Theta(X) \, \rightarrow \, \exists
Y\, \Psi(X,Y)]$ be a $\Pi^1_2$-problem. If $\mathsf{P}$ is provable
in $\mathsf{ACA}_0$, then there is an $n \in \omega$ such that
$\mathsf{ACA}_0$ proves $\forall X\,[\Theta(X) \, \rightarrow \,
\exists Y \in \Sigma^{0,X}_n\, \Psi(X,Y)]$.
\end{thm}

As mentioned above, this theorem implies for instance that
$\mathsf{ACA}_0 \nvdash \mathsf{RT}$, because Jockusch~\cite{Jockusch}
showed that for each $n \geq 2$, there is an instance of
$\mathsf{RT}^n_2$ (and hence of $\mathsf{RT}$) with no $\Sigma^0_n$
solutions. On the other hand, Jockusch also showed that every instance
of $\mathsf{RT}^n_k$ has a $\Pi^0_n$ solution, which implies that
every $\omega$-model of $\mathsf{ACA}_0$ is a model of $\mathsf{RT}$.

Notice that if we take $\mathsf{Q}$ to be the statement that for each
$X$, the Turing jump $X'$ exists, then the provability of $\mathsf{P}$
in $\mathsf{ACA}_0$ is equivalent to the provability of $\mathsf{Q}
\rightarrow \mathsf{P}$ in $\mathsf{RCA}_0$. As part of the proof of
Theorem~\ref{rcathm} in Section~\ref{compactnesssec}, we will prove a
theorem that is a direct generalization of
Theorem~\ref{wangthm}. Montalb\'an and Shore~\cite{MonSho} also
generalized this theorem, in a different way that is particularly
suited to problems where each instance has a unique solution, and is
indeed equivalent to ours in that case, but is not strong enough for
our purposes.

As an example of the application of Theorem~\ref{rcathm}, we will
obtain a simple proof that $\mathsf{RT}^2_2$ does not imply
$\mathsf{RT}^2_{<\infty}$, even over $\mathsf{RCA}_0$ together with all
$\Pi^1_1$ formulas true over the natural numbers.

\bigskip

Let $\Gamma$ be a class of formulas. Recall that $\mathsf{I}\Gamma$
is the axiom scheme stating that induction holds for formulas in
$\Gamma$. Recall also that the \emph{$\Gamma$-bounding axiom scheme}
$\mathsf{B}\Gamma$ consists of all formulas of the form
\[
\forall n\,[\forall i<n\, \exists k\, \phi(i,k) \, \rightarrow \,
\exists b\, \forall i<n\, \exists k \leq b\, \phi(i,k)]
\]
for each formula $\phi$ in $\Gamma$ such that $b$ is not free in
$\phi$. Note that $\phi$ is allowed to have parameters. The system
$\mathsf{RCA}_0 + \mathsf{B}\Sigma^0_2$, which is strictly
intermediate between $\mathsf{RCA}_0$ and $\mathsf{RCA}_0 +
\mathsf{I}\Sigma^0_2$, has been particularly prominent in reverse
mathematics. (In most cases, it is actually $\mathsf{B}\Pi^0_1$ that
is used, but $\mathsf{B}\Pi^0_1$ and $\mathsf{B}\Sigma^0_2$ are easily
seen to be equivalent over $\mathsf{RCA}_0$.) For example,
Hirst~\cite{Hirst} showed that $\textsf{RT}^1_{<\infty}$ is equivalent
to $\mathsf{B}\Sigma^0_2$ over $\mathsf{RCA}_0$.

In Section~\ref{compsec}, we will consider computable winning
strategies and the notion of generalized Weihrauch reducibility over
possibly nonstandard models. There is an intriguing connection here
with analogs of $\mathsf{RCA}_0$ for intuitionistic logic, first noted
in work of Kuyper~\cite{Kuyper}. We will comment on this connection
briefly in that section, but leave further work in this direction to a
follow-up paper. In Section~\ref{singlesec} we will consider
single-instance reducibilities such as computable and Weihrauch
reducibility in this context. Our results throughout will apply not
only to $\mathsf{RCA}_0$ but also to other systems at the level of
computable mathematics, including extensions of $\mathsf{RCA}_0$ by
first-order principles, such as $\mathsf{RCA}_0+\mathsf{I}\Sigma^0_n$
or $\mathsf{RCA}_0+\mathsf{B}\Sigma^0_n$, and also restrictions such
as $\mathsf{RCA}_0^*$, which roughly speaking is $\mathsf{RCA}_0$ with
$\Sigma^0_1$-induction replaced by $\Sigma^0_0$-induction.

In Sections~\ref{exsec} and~\ref{exsec2}, we will undertake a case
study in the analysis of mathematical principles under Weihrauch and
generalized Weihrauch reducibility over possibly nonstandard models,
by considering several principles that are equivalent over
$\mathsf{RCA}_0$ to $\Sigma^0_2$-bounding. We will see how this
framework allows us to uncover some hitherto hidden differences
between quite similar principles.

\section{Reduction games and provability}
\label{classicalsec}

In this section, we generalize Definition~\ref{generalgames} from
$\mathsf{RCA}_0$ to other axiom systems $\Gamma$, modify it as
described above, and prove a more general version of
Proposition~\ref{hjprop}. Of course, we cannot in general require
Player 1's moves to result in models of $\Gamma$, since it might be
the case that no structure of the form $M[X_0,\ldots,X_{n-1}]$ is a
model of $\Gamma$. However, we can require that Player 1 never make it
impossible for the model built by its moves to be extendable to a
model of $\Gamma$. Say that an $L_2$-structure $\mathcal M$ is
\emph{consistent with $\Gamma$} if it is contained in a model
$\mathcal N$ of $\Gamma$ with the same first-order part. (Note that if
$\mathcal M$ is countable, then we can require $\mathcal N$ to be
countable as well without changing the notion.)

The systems $\Gamma$ for which we will prove that winning strategies
for the following game correspond to provability over $\Gamma$ will
actually have the property that every structure consistent with
$\Gamma$ is in fact a model of $\Gamma$. The reason we give the
definition in the more general setting is that, when analyzing the
provability of $\mathsf{Q} \rightarrow \mathsf{P}$ in $\Gamma$, we
will also want to consider games over $\Gamma + \mathsf{Q}$. We will
see that doing so makes no difference in the case of general winning
strategies, but does in the case of computable winning strategies.

\begin{defn}
\label{redgamedef}
Let $\Gamma$ be a set of $L_2$-formulas and let $\mathsf{P}$ and
$\mathsf{Q}$ be $\Pi^1_2$-problems. The \emph{$\Gamma$-reduction game}
$G^{\Gamma}(\mathsf{Q} \rightarrow \mathsf{P})$ is a two-player game
played according to the following rules.
\begin{enumerate}[\rm (1)]

\item If at any point a player cannot make a move, the opponent wins.

\item If one of the players wins, the game ends.

\item On the first move, Player 1 plays a countable $L_1$-structure
$M$ and an $M$-instance $X_0$ of $\mathsf{P}$ such that $M[X_0]$ is
consistent with $\Gamma$. Then Player 2 either plays a solution to $X_0$
in $M[X_0]$ and wins, or plays an $M$-instance $Y_1$ of $\mathsf{Q}$
in $M[X_0]$.

\item For $n > 1$, on the $n$th move, Player 1 plays a solution
$X_{n-1}$ to the instance $Y_{n-1}$ of $\mathsf{Q}$ such that
$M[X_0,\ldots,X_{n-1}]$ is consistent with $\Gamma$. Then Player 2
either plays a solution to $X_0$ in $M[X_0,\ldots,X_{n-1}]$ and wins,
or plays an $M$-instance $Y_n$ of $\mathsf{Q}$ in
$M[X_0,\ldots,X_{n-1}]$.

\item If the game never ends then Player 1 wins.

\end{enumerate}
\end{defn}

We modify this game as follows.

\begin{defn}
Let $\Gamma$ be a set of $L_2$-formulas consistent with
$\Delta^0_1$-{\penalty0\hskip0pt\relax}comprehension, and let
$\mathsf{P}$ and $\mathsf{Q}$ be $\Pi^1_2$-problems. The
\emph{modified $\Gamma$-reduction game}
$\widehat{G}^{\Gamma}(\mathsf{Q} \rightarrow \mathsf{P})$ is a
two-player game played according to the following rules.
\begin{enumerate}[\rm (1)]

\item If at any point a player cannot make a move, the opponent wins.

\item If one of the players wins, the game ends.

\item On the first move, Player 1 plays a model $(M,S)$ of $\Gamma$
such that $M$ is countable and $S$ is closed under
$\Delta^0_1$-comprehension, and an $M$-instance $X_0$ of
$\mathsf{P}$ in $S$. Then Player 2 either plays a solution to $X_0$
in $M[X_0]$ and wins, or plays an $M$-instance $Y_1$ of $\mathsf{Q}$
in $M[X_0]$.

\item For $n > 1$, on the $n$th move, Player 1 plays a solution
$X_{n-1}$ to the instance $Y_{n-1}$ of $\mathsf{Q}$ in $S$. Then
Player 2 either plays a solution to $X_0$ in $M[X_0,\ldots,X_{n-1}]$
and wins, or plays an $M$-instance $Y_n$ of $\mathsf{Q}$ in
$M[X_0,\ldots,X_{n-1}]$.

\item If the game never ends then Player 1 wins.

\end{enumerate}
\end{defn}

If $\Gamma$ is consistent with $\Delta^0_1$-comprehension and $\Gamma
\nvdash \mathsf{Q} \rightarrow \mathsf{P}$, then Player 1 has winning
strategies in both of these games (as we will see in the second part
of the proof of Proposition~\ref{auxprop} below), but we cannot hope
in general that the same is the case for Player 2 if $\Gamma \vdash
\mathsf{Q} \rightarrow \mathsf{P}$, because of that player's
restriction to playing computably. However, if $\Gamma$ is
sufficiently well-behaved, then this is no longer an obstacle,
and we can obtain a generalization of Proposition~\ref{hjprop} with
essentially the same proof. The key property here is that all axioms
of $\Gamma$ other than $\Delta^0_1$-comprehension be $\Pi^1_1$. Of
course, this property holds of $\mathsf{RCA}_0$, as well as
commonly-studied first-order extensions such as $\mathsf{RCA}_0 +
\mathsf{I}\Sigma^0_n$ and $\mathsf{RCA}_0 + \mathsf{B}\Sigma^0_n$, and
restrictions such as $\mathsf{RCA}_0^*$.

In the proof, we will actually use the following properties, but it is
not difficult to show that they are equivalent to saying that $\Gamma$
is a consistent set of $L_2$-formulas consisting of
$\Delta^0_1$-comprehension together with a set of $\Pi^1_1$ formulas.
\begin{enumerate}

\item $\Gamma$ is a consistent set of $L_2$-formulas that
includes all instances of $\Delta^0_1$-comprehension.

\item If an $L_2$-structure is closed under $\Delta^0_1$-definability
and is consistent with $\Gamma$, then it is a model of $\Gamma$.

\item For every countable $L_1$-structure $M$ and
$X_0,X_1,\ldots \subseteq |M|$, if each $M[X_0,\ldots,X_n]$ is a
model of $\Gamma$, then so is their union $M[X_0,X_1,\ldots]$.

\end{enumerate}
The following simple but important result follows from these
properties.

\begin{lem}
\label{lem:P2canplay}
Let $\Gamma$ be a consistent extension of $\Delta^0_1$-comprehension
by $\Pi^1_1$ formulas. Let $\mathsf{P}$ and $\mathsf{Q}$ be
$\Pi^1_2$-problems. Let $M$ be an $L_1$-structure and $X_0,\ldots,X_n
\subseteq |M|$ be sets set such that $M[X_0,\ldots,X_n]$ is consistent
with $\Gamma$. If $\Gamma \vdash \mathsf{Q} \to \mathsf{P}$, then
either every instance of $\mathsf{P}$ in $M[X_0,\ldots,X_n]$ has a
solution in $M[X_0,\ldots,X_n]$, or else $\mathsf{Q}$ has an instance
in $M[X_0,\ldots,X_n]$.
\end{lem}

\begin{proof}
Fix $M$ and $X_0,\ldots,X_n$. Since the $L_2$-structure
$M[X_0,\ldots,X_n]$ is closed under $\Delta^0_1$-comprehension it is
in fact a model of $\Gamma$, as noted above. If $\mathsf{Q}$ has no
instance in $M[X_0,\ldots,X_n]$, then $M[X_0,\ldots,X_n]$ trivially
satisfies $\mathsf{Q}$. Hence, by assumption, $M[X_0,\ldots,X_n]$ also
satisfies $\mathsf{P}$. So every instance of $\mathsf{P}$ in
$M[X_0,\ldots,X_n]$ has a solution in $M[X_0,\ldots,X_n]$.
\end{proof}

Later on, when we prove a generalization of Theorem~\ref{rcathm}, we
will also need to assume that $\Gamma$ is strong enough to prove the
existence of a universal $\Sigma^0_1$ formula, but of course that
holds of all systems we normally study in reverse mathematics.

We should also expect $\Gamma$ and $\Gamma + \mathsf{Q}$ to behave
similarly here, since there is no difference between saying that
$\Gamma \vdash \mathsf{Q} \rightarrow \mathsf{P}$ and saying that
$\Gamma + \mathsf{Q} \vdash \mathsf{Q} \rightarrow \mathsf{P}$. This
fact will be of interest below when we consider computable winning
strategies.

Proposition~\ref{hjprop} can be generalized as follows. Notice that
of the four games $G^{\Gamma}(\mathsf{Q} \rightarrow \mathsf{P})$,
$G^{\Gamma + \mathsf{Q}}(\mathsf{Q} \rightarrow \mathsf{P})$,
$\widehat{G}^{\Gamma}(\mathsf{Q} \rightarrow \mathsf{P})$, and
$\widehat{G}^{\Gamma + \mathsf{Q}}(\mathsf{Q} \rightarrow
\mathsf{P})$, the first is the hardest one for Player 2 to win, while
the last is the hardest one for Player 1 to win.

\begin{prop}
\label{auxprop}
Let $\Gamma$ be a consistent extension of $\Delta^0_1$-comprehension
by $\Pi^1_1$ formulas. Let $\mathsf{P}$ and $\mathsf{Q}$ be
$\Pi^1_2$-problems. If $\Gamma \vdash \mathsf{Q} \rightarrow
\mathsf{P}$ then Player 2 has a winning strategy for
$G^{\Gamma}(\mathsf{Q} \rightarrow \mathsf{P})$ (and hence for each of
the three other games above). Otherwise, Player 1
has a winning strategy for $\widehat{G}^{\Gamma +
\mathsf{Q}}(\mathsf{Q} \rightarrow \mathsf{P})$ (and hence for each of
the three other games above).
\end{prop}

\begin{proof}
If $\Gamma \vdash \mathsf{Q} \rightarrow \mathsf{P}$ then Player 2 can
play according to the following strategy. Let $M$ be the
$L_1$-structure played by Player 1 on its first move.  At the $n$th
move, if Player 2 has a legal winning move, Player 2 makes that
move. Otherwise, it lets $Y_{n,0},Y_{n,1},\ldots$ be all $M$-instances
of $\mathsf{Q}$ in $M[X_0,\ldots,X_{n-1}]$, where $X_0,\ldots,X_{n-1}$
are Player 1's first $n$ moves. For the least pair $\langle m,i
\rangle$ with $m \leq n$ for which Player 2 has not yet acted, it then
acts by playing $Y_{m,i}$ (to which Player 1 must reply with a
solution to $Y_{m,i}$). Note that Player 2 always has some legal move,
by Lemma \ref{lem:P2canplay}. Suppose Player 2 never has a winning
move, and Player 1 never fails to have a legal move. By our
assumptions on $\Gamma$, each $M[X_0,\ldots,X_{n-1}]$ is a model of
$\Gamma$, and hence so is their union $M[X_0,X_1,\ldots]$. But Player
2's strategy ensures that this structure is also a model of
$\mathsf{Q}$, so it must also be a model of $\mathsf{P}$, and hence
must contain a solution to $X_0$. This solution is in
$M[X_0,\ldots,X_{n-1}]$ for some $n$, which gives Player 2 a winning
$n$th move.

If $\Gamma \nvdash \mathsf{Q} \rightarrow \mathsf{P}$ then let $(M,S)$
be a model of $\Gamma + \mathsf{Q} + \neg \mathsf{P}$ and let $X_0$ be
an $M$-instance of $\mathsf{P}$ in $S$ with no solution in $S$. Since
$(M,S)$ is a model of $\Gamma$, it is closed under
$\Delta^0_1$-definability, so as long as Player 1's moves stay inside
$S$, so must Player 2's moves. Furthermore, the fact that $(M,S)$ is a
model of $\mathsf{Q}$ implies that, as long as Player 2's moves stay
inside $S$, Player 1 will always be able to reply with moves that stay
inside $S$. So Player 1 can simply begin by playing $(M,S)$ and $X_0$,
and then keep playing elements of $S$, which ensures that the game
never ends (unless Player 2 cannot make its first move, in which case
it loses immediately).
\end{proof}

\begin{rem}
We can extend the above framework beyond extensions of
$\Delta^0_1$-comprehension by $\Pi^1_1$ formulas. Let us consider
$\mathsf{ACA}_0$, for instance. If we redefine $M[X_0,\ldots,X_{n-1}]$
by replacing $\Delta^0_1$-definability by arithmetic definability,
then use this new definition in the definitions of the
$\Gamma$-reduction game and the modified $\Gamma$-reduction game, then
Proposition~\ref{auxprop} carries through essentially unchanged.

There is nothing particularly special about this
$\Gamma=\mathsf{ACA}_0$ case. All we need is the existence of a
smallest model $M[X_0,\ldots,X_{n-1}]$ of $\Gamma$ with first-order
part $M$ containing $X_0,\ldots,X_{n-1} \subseteq |M|$ (if there is
any such model at all), and the requirement that then $\bigcup_n
M[X_0,\ldots,X_{n-1}]$ is also a model of $\Gamma$ (which will happen
if $\Gamma$ is $\Pi^1_2$-axiomatizable). For systems $\Gamma$ that do
not have such minimal models, such as $\mathsf{WKL}_0$, we can still
extend these ideas by redefining our games in a way that does not
affect our results when applied to systems that do have minimal
models. For example,  $\widehat{G}^{\Gamma}(\mathsf{Q} \rightarrow
\mathsf{P})$ can now be played according to the following rules.
\begin{enumerate}[\rm (1)]

\item If at any point a player cannot make a move, the opponent wins.

\item If one of the players wins, the game ends.

\item On the first move, Player 1 plays a model $(M,S)$ of $\Gamma$
with $M$ countable, an $M$-instance $X_0$ of $\mathsf{P}$ in $S$, and
a submodel $(M,S_0)$ of $(M,S)$ containing $X_0$. Then Player 2
either plays a solution to $X_0$ in $(M,S_0)$ and wins, or plays an
$M$-instance $Y_1$ of $\mathsf{Q}$ in $(M,S_0)$.

\item For $n > 1$, on the $n$th move, Player 1 plays a solution
$X_{n-1}$ to the instance $Y_{n-1}$ of $\mathsf{Q}$ in $S$ and a
submodel $(M,S_{n-1})$ of $(M,S)$ containing $X_{n-1}$. Then Player
2 either plays a solution to $X_0$ in $(M,S_{n-1})$ and wins, or
plays an $M$-instance $Y_n$ of $\mathsf{Q}$ in $(M,S_{n-1})$.

\item If the game never ends then Player 1 wins.

\end{enumerate}

Theorem~\ref{classicalmainthm} below remains true for
$\mathsf{ACA}_0$, for instance, since in Theorem~\ref{classicalauxthm}
we can replace the $e$th Turing functional by the $e$th arithmetical
functional. It is not clear how generally
Theorem~\ref{classicalmainthm} holds for other systems, but we will
not pursue this further generalization of our framework here.
\end{rem}

\section{Reduction games and compactness}
\label{compactnesssec}

As mentioned in the introduction, we can improve on
Proposition~\ref{auxprop} by showing that a certain kind of
compactness theorem holds, with the very mild extra assumption that
$\Gamma$ proves the existence of a universal $\Sigma^0_1$ formula,
i.e., that there is a $\Sigma^0_1$ formula $\theta(e,n,X)$ such that
for every $\Sigma^0_1$ formula $\phi(e,n,X)$, we have $\Gamma \vdash
\forall e\, \exists i\, \forall n\, \forall X\,(\theta(i,n,X)
\leftrightarrow \phi(e,n,X))$. In this case, we assume we have fixed
such a $\theta$ and a bijective pairing function $\langle \cdot,\cdot
\rangle$, and write $Y = \Phi_e^X$ to mean that for $e=\langle i,j
\rangle$, we have $\forall n\, [\theta(i,n,X) \leftrightarrow
\neg \theta(j,n,X)]$ and $\forall n\, [n \in Y \leftrightarrow
\theta(i,n,X)]$.

The following result, which we will prove in this section, has
Theorem~\ref{rcathm} as a special case.

\begin{thm}
\label{classicalmainthm}
Let $\Gamma$ be a consistent extension of $\Delta^0_1$-comprehension
by $\Pi^1_1$ formulas that proves the existence of a universal
$\Sigma^0_1$ formula. Let $\mathsf{P}$ and $\mathsf{Q}$ be
$\Pi^1_2$-problems. If $\Gamma \vdash \mathsf{Q} \rightarrow
\mathsf{P}$ then there is an $n$ such that Player 2 has a winning
strategy for $\widehat{G}^{\Gamma}(\mathsf{Q} \rightarrow \mathsf{P})$
(and hence for $\widehat{G}^{\Gamma +
\mathsf{Q}}(\mathsf{Q} \rightarrow \mathsf{P})$)
that ensures victory in at most $n$ many moves. Otherwise, Player 1
has a winning strategy for $\widehat{G}^{\Gamma +
\mathsf{Q}}(\mathsf{Q} \rightarrow \mathsf{P})$ (and hence for
$\widehat{G}^{\Gamma}(\mathsf{Q} \rightarrow \mathsf{P})$).
\end{thm}

Notice that if the formulas added to $\Delta^0_1$-comprehension to
obtain $\Gamma$ are true over the standard natural numbers, then a
winning strategy for Player 2 for $\widehat{G}^{\Gamma}(\mathsf{Q}
\rightarrow \mathsf{P})$ that ensures victory in at most $n$ many
moves also yields a winning strategy for Player 2 for $G(\mathsf{Q}
\rightarrow \mathsf{P})$ that ensures victory in at most $n$ many
moves, since a run of the latter game is a special case of a run of
the former game in which Player 1 begins by playing the model
$(\omega,\mathcal P(\omega))$. Thus it is not a coincidence that all
the examples we have of situations in which $G(\mathsf{Q} \rightarrow
\mathsf{P})$ can be won by Player 2, but not in a number of moves
bounded ahead of time, are ones in which $\mathsf{P}
\leq\sub{$\omega$} \mathsf{Q}$ but $\mathsf{RCA}_0 \nvdash \mathsf{Q}
\rightarrow \mathsf{P}$. In fact, the following stronger fact holds,
where, as defined in~\cite{HirJoc}, $\mathsf{P}
\leq\sub{$\omega$}\sup{$n$} \mathsf{Q}$ means that Player 2 has a
winning strategy for $G(\mathsf{Q} \rightarrow \mathsf{P})$ that
ensures victory in at most $n+1$ many moves.

\begin{cor}
Let $\Gamma$ consist of $\mathsf{RCA}_0$ together with all $\Pi^1_1$
formulas true over the natural numbers. If $\mathsf{P}
\nleq\sub{$\omega$}\sup{$n$} \mathsf{Q}$ for all $n$, then $\Gamma
\nvdash \mathsf{Q} \rightarrow \mathsf{P}$.
\end{cor}

Notice that the $\Gamma$ in this corollary includes full arithmetical
induction. An interesting example of the application of this corollary
is to take $\mathsf{Q}$ to be $\mathsf{RT}^2_2$ and $\mathsf{P}$ to be
$\mathsf{RT}^2_{<\infty}$. Cholak, Jockusch,
and Slaman~\cite{CJS} showed that $\mathsf{RCA}_0 \nvdash
\mathsf{RT}^2_k \rightarrow \mathsf{RT}^2_{<\infty}$ for all $k$, but the proof
relies on a difference between the first-order parts of these two
principles, and hence does not work if we add arithmetical induction
to $\mathsf{RCA}_0$. (Note that, with full induction,
$\mathsf{RT}^2_{<\infty}$ does in fact follow from $\mathsf{RT}^2_2$.)
Patey~\cite{Patey2} showed that $\mathsf{RT}^2_{<\infty}
\nleq\sub{$\omega$}\sup{$n$} \mathsf{RT}^2_k$ for all $n$ and $k$, so
we have the following.

\begin{cor}
\label{rt2cor}
Let $\Gamma$ consist of $\mathsf{RCA}_0$ together with all
$\Pi^1_1$ formulas true over the natural numbers. Then $\Gamma \nvdash
\mathsf{RT}^2_k \rightarrow \mathsf{RT}^2_{<\infty}$ for all $k$.
\end{cor}

\noindent We learned from Yokoyama [personal communication] that he and
Slaman have recently noticed that this corollary can also be obtained by a
more direct model-theoretic argument, still using Patey's result.

The proof of Theorem~\ref{classicalmainthm} will use the following
result, which is of independent interest as a generalization of
Theorem~\ref{wangthm}.

\begin{thm}
\label{classicalauxthm}
Let $\Gamma$ be a consistent extension of $\Delta^0_1$-comprehension
by $\Pi^1_1$ formulas that proves the existence of a universal
$\Sigma^0_1$ formula.  Let $\mathsf{P}$ and $\mathsf{Q}$ be
$\Pi^1_2$-problems. For $n \in \omega$, let $\Theta_n(e_0, \ldots,
e_n,X_0, \ldots, X_n, Y_0, \ldots, Y_n)$ be a formula asserting that
\begin{multline*}
 \text{ if }X_0\text{ is a } \mathsf{P}\text{-instance then }(Y_0 = \Phi_{e_0}^{X_0} \wedge (\text{either }Y_0\text{ is a solution to }X_0\text{ or} \\
  (Y_0\text{ is a }\mathsf{Q}\text{-instance and if }X_1\text{ is a solution to }Y_0\text{ then }(Y_1 = \Phi_{e_1}^{X_0 \oplus X_1} \wedge \\(\text{either }Y_1\text{ is a solution to }X_0\text{ or} \\   (Y_1\text{ is a }\mathsf{Q}\text{-instance and if }X_2\text{ is a solution to }Y_1\text{ then }(Y_2 = \Phi_{e_2}^{X_0 \oplus X_1 \oplus X_2} \wedge \\(\text{either }Y_2\text{ is a solution to } X_0\text{ or }\ldots \\
                \vdots \\
                 \ldots (Y_n = \Phi_{e_n}^{X_0 \oplus \cdots \oplus
                   X_n} \wedge Y_n\text{ is a solution to
                 }X_0))\cdots),
\end{multline*}
and let $\Delta_n$ be 
\[
\forall X_0\, \exists e_0, Y_0\, \forall X_1\, \exists e_1,
Y_1 \cdots \forall X_n\, \exists e_n, Y_n\, \Theta_n(e_0, \ldots, e_n,X_0, \ldots, X_n,
Y_0, \ldots, Y_n).
\]
If $\Gamma \vdash \mathsf{Q} \to \mathsf{P}$, then
 there exists
an $n \in \omega$ such that $\Gamma \vdash \Delta_n$.
\end{thm}

\begin{proof}
Suppose that $\Gamma \vdash \mathsf{Q} \rightarrow \mathsf{P}$ but
$\Gamma \vdash \neg\Delta_n$ for all $n$. Extend $L_2$ to include a
function symbol $f$ from first-order objects to second-order
objects. Call this new language $L_2'$. Let $\langle
\cdot,\ldots,\cdot \rangle$ be a fixed numbering scheme for finite
tuples of numbers.

For each $n$, there is a model $\mathcal M=(M,S)$ of
$\Gamma+\neg\Delta_n$. We can turn $\mathcal M$ into an
$L_2'$-structure by defining the interpretation $f^{\mathcal M}$ by
recursion as follows.

There is an $X_0 \in S$ such that
\begin{multline*}
\mathcal M \vDash \forall e_0, Y_0\, \exists X_1\,
\forall e_1, Y_1 \cdots \exists X_n\, \forall e_n, Y_n\, \neg
\Theta_n(e_0, \ldots, e_n,\\
X_0, \ldots, X_n, Y_0, \ldots, Y_n).
\end{multline*}
Let $f^{\mathcal M}(\langle{}\rangle)=X_0$.

Assume we have defined $f^{\mathcal M}(\langle e_0,\ldots,
e_{j-1}\rangle)$, where $j < n$, and have also defined
$Y_{\langle e_0 \rangle},Y_{\langle e_0,e_1 \rangle},\ldots,Y_{\langle
e_0,\ldots,e_{j-1}\rangle} \in S$ so that 
\begin{multline*}
\mathcal M \vDash \forall e_j, Y_j \,
\exists X_{j+1}\, \forall e_{j+1}, Y_{j+1} \exists X_{j+2} \cdots \exists X_n\, \forall e_n, Y_n\,
\neg \Theta_n(e_0, \ldots, e_n,\\
f^{\mathcal
  M}(\langle{}\rangle),f^{\mathcal M}(\langle e_0 \rangle),
\ldots,f^{\mathcal M }(\langle e_0,\ldots,e_{j-1}\rangle),X_{j+1}\ldots,
X_n,\\
Y_{\langle e_0 \rangle},Y_{\langle e_0,e_1 \rangle},\ldots,Y_{\langle
e_0,\ldots,e_{j-1}\rangle}, Y_j, \ldots, Y_n).
\end{multline*}
 Given $e_j \in M$, let
$Y_{\langle e_0,\ldots,e_j\rangle}=\Phi_{e_j}^{(f^{\mathcal M}(\langle{}\rangle) \oplus f^{\mathcal
M}(\langle e_0 \rangle) \oplus \cdots \oplus f^{\mathcal M
}(\langle e_0,\ldots,e_{j-1}\rangle))} \in S$. Then there is an $X_{j+1}
 \in S$
such that 
\begin{multline*}
\mathcal M \vDash \forall e_{j+1}, Y_{j+1} \,
\exists X_{j+2}\, \forall e_{j+2}, Y_{j+2} \exists X_{j+3} \cdots \exists X_n\, \forall e_n, Y_n\,
\neg \Theta_n(e_0, \ldots, e_n,\\
f^{\mathcal
M}(\langle{}\rangle),f^{\mathcal M}(\langle e_0 \rangle),
\ldots,f^{\mathcal M }(\langle
e_0,\ldots,e_{j-1}\rangle),X_{j+1}\ldots,X_n,\\
Y_{\langle e_0 \rangle},Y_{\langle e_0,e_1 \rangle},\ldots,Y_{\langle
e_0,\ldots,e_j\rangle},Y_{j+1}, \ldots, Y_n). 
\end{multline*}
Let $f^{\mathcal M}(\langle e_0,\ldots,e_j
\rangle)=X_{j+1}$.

Having defined $f^{\mathcal M}$ on all $\langle e_0,\ldots,e_i
\rangle$ for $i \leq n$, let $f^{\mathcal M}(x)=\emptyset$ for all
other $x \in M$. 

Let
\begin{multline*}
\Psi_k \equiv \forall e_0,Y_0 \, \cdots\, \forall e_k,Y_k\,
\neg \Theta_k(e_0, \ldots, e_k,\\
f(\langle{}\rangle),f(\langle e_0 \rangle),
\ldots,f(\langle e_0,\ldots,e_k\rangle),Y_0, \ldots, Y_k).
\end{multline*}
Then
$(\mathcal M;f^{\mathcal M}) \vDash \Psi_n$ by the definition of
$f^{\mathcal M}$. It is easy to see that this fact implies that
$(\mathcal M;f^{\mathcal M}) \vDash \Psi_k$ for all $k \leq n$.

Thus every set $\Gamma \cup \{\Psi_0,\ldots,\Psi_n\}$ is satisfiable,
and hence so is the union $\Gamma \cup \{\Psi_0,\Psi_1,\ldots\}$. Let $\mathcal
N=(N,T)$ be a model of this set. Now we have a winning
strategy for Player 1 for $G^{\Gamma}(\mathsf{Q} \rightarrow
\mathsf{P})$: Player 1 begins by playing $N$ and $f^{\mathcal
N}(\langle{}\rangle)$, and if
$e_0,\ldots,e_{n-1}$ are indices for Player 2's first $n$
moves, then Player 1 plays $f^{\mathcal N}(\langle e_0,\ldots,e_{n-1}\rangle)$ on
its next move. By the definition of $\Psi_n$, Player 2 can never play
a solution to $f^{\mathcal N}(\langle{}\rangle)$.

But by Proposition~\ref{auxprop} and our assumption
that $\Gamma \vdash \mathsf{Q} \rightarrow \mathsf{P}$, Player 2 must
have a winning strategy for $G^{\Gamma}(\mathsf{Q} \rightarrow
\mathsf{P})$, so we have a contradiction.
\end{proof}

\begin{proof}[Proof of Theorem~\ref{classicalmainthm}]
We use the notation of Theorem~\ref{classicalauxthm}.  By
Proposition~\ref{auxprop}, it is enough to show that if $\Gamma \vdash
\mathsf{Q} \rightarrow \mathsf{P}$ then there is an $n$ such that
Player 2 has a winning strategy for $\widehat{G}^{\Gamma}(\mathsf{Q}
\rightarrow \mathsf{P})$ that ensures victory in at most $n$ many
moves. So suppose that $\Gamma \vdash \mathsf{Q} \rightarrow
\mathsf{P}$. Let $n$ be as in Theorem~\ref{classicalauxthm}.

Player 2 can play as follows. Let $\mathcal M=(M,S)$
be the model of $\Gamma$ played by Player 1 on its first move. Since
$\mathcal M$ is a model of $\Gamma$, it is also a model of
$\Delta_n$. Let $X_0$ be Player 1's first move. Since $X_0$ is in $S$,
there are $e_0 \in M$ and $Y_0 \in S$ such that $\mathcal M$ satisfies
\[
\forall X_1\, \exists e_1, Y_1\, \forall X_2\, \exists e_2,
Y_2 \cdots \forall X_n\, \exists e_n, Y_n\, \Theta_n(e_0, \ldots,
e_n,X_0, \ldots, X_n, Y_0, \ldots, Y_n).
\]
Now Player 2 plays $Y_0$. Let $X_1$ be Player 1's next move. Then
there are $e_1 \in M$ and $Y_1 \in S$ such that
$\mathcal M$ satisfies
\[
\forall X_2\, \exists e_2, Y_2\, \forall X_3\, \exists e_3,
Y_3 \cdots \forall X_n\, \exists e_n, Y_n\, \Theta_n(e_0, \ldots,
e_n,X_0, \ldots, X_n, Y_0, \ldots, Y_n).
\]
Now Player 2 plays $Y_1$.

Continuing in this way, by the definition of $\Delta_n$, some $Y_i$
with $i \leq n$ must be a solution
to $X_0$, and thus this strategy ensures victory by Player 2 in at
most $n+1$ many moves.
\end{proof}

We do not know whether Theorem~\ref{classicalmainthm} holds for
$G^\Gamma(\mathsf{Q} \rightarrow \mathsf{P})$ in general, but
normally, if $\Gamma \vdash \mathsf{Q} \rightarrow \mathsf{P}$ then
the proof allows us to obtain a winning strategy for Player 2 in
$\widehat{G}^\Gamma(\mathsf{Q} \rightarrow \mathsf{P})$ (and even in
$G^\Gamma(\mathsf{Q} \rightarrow \mathsf{P})$) that is relatively easy
to describe. (The special case of computable winning strategies will
be discussed in Section~\ref{compsec}.) In such cases, we
can show that there is an $n$ such that this particular winning
strategy allows Player 2 to win in at most $n$ many moves, not just in
$\widehat{G}^\Gamma(\mathsf{Q} \rightarrow \mathsf{P})$ but in fact in
$G^\Gamma(\mathsf{Q} \rightarrow \mathsf{P})$. Here we are
thinking of strategies that are first-order definable, but we need to
take into account the possibility that there might not be a unique
choice of move at a given point (keeping in mind that the idea of
choosing the least among the indices of equally good moves is not
always available when working over nonstandard models). 

\begin{defn}
\label{accordingdef}
Let $\Gamma$ be a consistent set of $L_2$-formulas and let
$\Lambda(X,n,e)$ be an arithmetic formula. Say that Player 2 plays a
run of $G^\Gamma(\mathsf{Q} \rightarrow \mathsf{P})$ or
$\widehat{G}^\Gamma(\mathsf{Q} \rightarrow \mathsf{P})$
\emph{according to $\Lambda$} if given Player 1's first $n$ moves, $M$
(or $(M,S)$) and $X_0,\ldots,X_{n-1} \subseteq M$, Player 2 plays
$\Phi_e^{X_0 \oplus \cdots \oplus X_{n-1}}$ for some $e \in M$ such
that $M[X_0,\ldots,X_{n-1}] \vDash \Lambda(X_0 \oplus \cdots \oplus
X_{n-1},n-1,e)$.
\end{defn}

\begin{thm}
\label{defthm}
Let $\Gamma$ be a consistent extension of $\Delta^0_1$-comprehension
that proves the existence of a universal $\Sigma^0_1$ formula. Let
$\mathsf{P}$ and $\mathsf{Q}$ be $\Pi^1_2$-problems and
$\Lambda$ be an arithmetic formula such that Player 2 wins any
run of $\widehat{G}^\Gamma(\mathsf{Q} \rightarrow \mathsf{P})$ that
it plays according to $\Lambda$. Then there is an $n$ such that Player
2 wins any run of $G^\Gamma(\mathsf{Q} \rightarrow \mathsf{P})$ that
it plays according to $\Lambda$ in at most $n$ many moves.
\end{thm}

\begin{proof}
Let $\Theta_n$ be as in Theorem~\ref{classicalauxthm}. Let $\Xi_n$ be
a formula asserting that, for all $i \leq n$, if $X_0$ is a
$\mathsf{P}$-instance and no $Y_j$ with $j<i$ is a solution to $X_0$,
then $\Lambda(X_0 \oplus \cdots \oplus X_i,i,e_i)$. Let
$\widehat{\Theta}_n$ be $\Xi_n \rightarrow \Theta_n$, and let
$\Omega_n$ be
\begin{multline*}
\forall X_0\, \forall e_0\, \exists Y_0\, \forall X_1\, \forall e_1\,
\exists Y_1 \cdots \forall X_n\, \forall e_n\, \exists Y_n\\
\widehat{\Theta}_n(e_0,\ldots,e_n,X_0,\ldots,X_n,Y_0,\ldots,Y_n).
\end{multline*}
Suppose there is a run of $G^\Gamma(\mathsf{Q} \rightarrow
\mathsf{P})$ such that Player 2 plays according to $\Lambda$ but does
not win within $n$ moves. Let $M$ and $X_0,\ldots,X_{n-1}$ be Player
1's first $n$ moves in that run. Then $M[X_0,\ldots,X_{n-1}]$ can be
extended to a model $(M,S)$ of $\Gamma$, and in that model,
$\Omega_{n-1}$ does not hold. Thus, to establish the theorem, it is
enough to show that $\Gamma \vdash \Omega_n$ for some $n$.

Assume for a contradiction that $\Gamma \nvdash \Omega_n$ for all
$n$. Expand $L_2$ by adding first-order constant symbols
$c_0,c_1,\ldots$ and second-order constant symbols
$C_0,C_1,\ldots$\hspace{2pt}\nolinebreak. Then a compactness argument
as in the proof of Theorem~\ref{classicalauxthm} shows that there is a
model $\mathcal M$ of $\Gamma$ and interpretations $c_0^{\mathcal
M},c_1^{\mathcal M},\ldots$ and $C_0^{\mathcal M},C_1^{\mathcal
M},\ldots$ such that each $\Phi_{c_n^{\mathcal M}}^{C_0^{\mathcal M}
\oplus \cdots \oplus C_n^{\mathcal M}}$ is total in $\mathcal M$,
and $\mathcal M$ together with these interpretations satisfies
\[
\neg\widehat{\Theta}_n(c_0,\ldots,c_n,C_0,\ldots,C_n,\Phi_{c_0}^{C_0},\ldots,\Phi_{c_n}^{C_0
  \oplus \ldots \oplus C_n})
\] 
for all $n$. But then there is a run of
$\widehat{G}^\Gamma(\mathsf{Q} \rightarrow \mathsf{P})$ in which
Player 2 plays according to $\Lambda$ but does not win, namely the one
in which Player 1 begins by playing $\mathcal M$, then at each move
plays $C_n^{\mathcal M}$, and Player 2 responds with $\Phi_{c_n^{\mathcal M}}^{C_0^{\mathcal M}
\oplus \cdots \oplus C_n^{\mathcal M}}$, which contradicts our
hypothesis.
\end{proof}

For $\Gamma$ is as in Theorem~\ref{classicalmainthm}, write $\Gamma
\vdash\sup{$n$} \mathsf{Q} \rightarrow \mathsf{P}$ to mean that Player
2 has a winning strategy for $\widehat{G}^{\Gamma}(\mathsf{Q}
\rightarrow \mathsf{P})$ that ensures victory in at most $n+1$ many
moves. Then the first part of the theorem can be restated as $\Gamma
\vdash \mathsf{Q} \rightarrow \mathsf{P} \; \Rightarrow \; \exists n\,
[\Gamma \vdash\sup{$n$} \mathsf{Q} \rightarrow \mathsf{P}]$. The idea
behind this notation is that we can see the least $n$ such that
$\Gamma \vdash\sup{$n$} \mathsf{Q} \rightarrow \mathsf{P}$ as a
measure of the number of applications of $\mathsf{Q}$ needed to prove
$\mathsf{P}$ over $\Gamma$. The $n=0$ case is equivalent to $\Gamma
\vdash \mathsf{P}$. We will discuss the $n=1$ case in
Section~\ref{singlesec}, but make the following remark for now.

\begin{rem}
\label{01rem}
Recall that $\mathsf{P} \leq\sub{$\omega$}\sup{$n$} \mathsf{Q}$ means
that Player 2 has a winning strategy for $G(\mathsf{Q} \rightarrow
\mathsf{P})$ that ensures victory in at most $n+1$ many
moves. Hirschfeldt and Jockusch~\cite{HirJoc} stated that $\mathsf{P}
\leq\sub{$\omega$}\sup{$1$} \mathsf{Q}$ is equivalent to $\mathsf{P}
\leq\sub{c} \mathsf{Q}$, but that is not quite correct, because if
$\mathsf{P}$ is computably true (i.e., if $\mathsf{P}
\leq\sub{$\omega$}\sup{$0$} \mathsf{Q}$) but has an instance that does
not compute any instance of $\mathsf{Q}$, then $\mathsf{P}
\leq\sub{$\omega$}\sup{$1$} \mathsf{Q}$ but $\mathsf{P} \nleq\sub{c}
\mathsf{Q}$. (The same point was made in the context of Weihrauch
reducibility by Brattka, Gherardi, and Pauly~\cite[Section 3]{BGP}.)
As this fairly uninteresting case is the only in which the two notions
differ, however, we can generally ignore the distinction. We mention
it, and make the following remarks, only because an analogous
situation will be relevant below.

We can define $\mathsf{P} \leq\sub{$\omega$}\sup{$\mathord{=}n$}
\mathsf{Q}$ to mean that Player 2 has a winning strategy for
$G(\mathsf{Q} \rightarrow \mathsf{P})$ that ensures victory in exactly
$n+1$ many moves. Then $\mathsf{P} \leq\sub{c} \mathsf{Q}$ is
equivalent to $\mathsf{P} \leq\sub{$\omega$}\sup{$\mathord{=}1$}
\mathsf{Q}$. This definition is not otherwise very useful, though,
because if Player 2 can win $G(\mathsf{Q} \rightarrow \mathsf{P})$ in
$m \geq 2$ many moves, then it can also win that game in $k$ many
moves for any $k>m$, simply by repeating its first move until it is
ready to win, except in the case in which Player 2's first move is an
instance of $\mathsf{Q}$ with no solution (and in this context we are
generally not interested in problems that are false over $\omega$ as
statements of second-order arithmetic).

Note also that $\mathsf{P} \leq\sub{$\omega$}\sup{$n$} \mathsf{Q}$ is
not quite equivalent to $\exists m \leq n\, [\mathsf{P}
\leq\sub{$\omega$}\sup{$\mathord{=}m$} \mathsf{Q}]$, again because
of $1$-move runs. For example, let $\mathsf{P}$ be the
$\Pi^1_2$-problem whose instances are $\emptyset$ and $\emptyset'$,
with unique solutions $\emptyset$ and $\emptyset''$, respectively; and
let $\mathsf{Q}$ be the $\Pi^1_2$-problem whose only instance is
$\emptyset'$, with unique solution $\emptyset''$. If Player 1 begins
by playing $\emptyset'$, then Player 2 cannot win immediately, but can
play $\emptyset'$, to which Player 1 must reply with $\emptyset''$, at
which point Player 2 wins by playing $\emptyset''$. So in this case,
Player 2 wins in $2$ moves. However, if Player 1 plays $\emptyset$,
then Player 2 has only one legal move, namely the winning move
$\emptyset$. Thus $\mathsf{P} \leq\sub{$\omega$}\sup{$1$} \mathsf{Q}$,
but the first case shows that $\mathsf{P}
\nleq\sub{$\omega$}\sup{$\mathord{=}0$} \mathsf{Q}$, while the second
case shows that $\mathsf{P} \nleq\sub{$\omega$}\sup{$\mathord{=}1$}
\mathsf{Q}$.

Similar considerations hold for the notion of $\mathsf{P}
\leq\sub{gW}\sup{$n$} \mathsf{Q}$ introduced in~\cite{HirJoc}, and for
$\Gamma \vdash\sup{$n$} \mathsf{Q} \rightarrow \mathsf{P}$. One way
around these issues is to replace $\mathsf{Q}$ with the problem
$\widehat{\mathsf{Q}}$ where an instance is either $\{0\} \cup \{n+1 :
n \in X\}$ for an instance $X$ of $\mathsf{Q}$, with a solution to
this instance being any solution to $X$; or $\emptyset$, with the only
solution being $\emptyset$ (although if we allow problems $\mathsf{Q}$
that have instances with no solutions, we might still have $\mathsf{P}
\leq\sub{gW}^n \widehat{\mathsf{Q}}$ but not have $\exists m \leq n\,
[\mathsf{P} \leq\sub{gW}\sup{$\mathord{=}m$} \widehat{\mathsf{Q}}]$,
because a computable winning strategy might not be able to tell when
it is about to play an instance of $\mathsf{Q}$ with no solution, and
thus instantly win).
\end{rem}

The definition of $\Gamma \vdash\sup{$n$} \mathsf{Q} \rightarrow
\mathsf{P}$ was made in~\cite{HirJoc} (for $\Gamma=\mathsf{RCA}_0$),
but with $G^{\Gamma}(\mathsf{Q} \rightarrow \mathsf{P})$ in place of
$\widehat{G}^{\Gamma}(\mathsf{Q} \rightarrow \mathsf{P})$. We have
chosen our definition in light of Theorem~\ref{classicalmainthm}, but
at least in natural cases, there should be no difference, as shown by
the following fact.

\begin{prop}
\label{defprop}
Let $\Gamma$ be a consistent extension of $\Delta^0_1$-comprehension
that proves the existence of a universal $\Sigma^0_1$ formula. Let
$\mathsf{P}$ and $\mathsf{Q}$ be $\Pi^1_2$-problems and $\Lambda$ be
an arithmetic formula such that Player 2 wins any run of
$\widehat{G}^\Gamma(\mathsf{Q} \rightarrow \mathsf{P})$ that it plays
according to $\Lambda$ in at most $n$ many moves. Then Player 2 wins
any run of $G^\Gamma(\mathsf{Q} \rightarrow \mathsf{P})$ that it
plays according to $\Lambda$ in at most $n$ many moves.
\end{prop}

\begin{proof}
In the notation of the proof of Theorem~\ref{defthm}, it is easy to
see that $\Gamma \vdash \Omega_{n-1}$, and hence Player 2 has a winning
strategy for $G^{\Gamma}(\mathsf{Q} \rightarrow \mathsf{P})$ that
ensures victory in at most $n$ many moves as in that proof.
\end{proof}

\begin{rem}
Hirst and Mummert~\cite{HirMum} discussed a different potential form
of instance-counting, based on a notion of proving a $\Pi^1_2$
principle $\mathsf{P}$ with one typical use of another $\Pi^1_2$
principle $\mathsf{Q}$ in a system $\Gamma$. While the definition of
that notion in their paper is not quite correct~[Hirst and Mummert,
personal communication], its main significance is that it allowed
them to conclude that, in cases of interest, $\Gamma$ then proves that
for every instance $X$ of $\mathsf{P}$, there is an instance $Y$ of
$\mathsf{Q}$ such that if $Y$ has a solution then so does $X$. While
their paper is mostly concerned with intuitionistic logic, they also
gave examples showing that this notion does not seem useful in the
context of classical logic. In particular they showed how
$\mathsf{RT}^2_4$ can be obtained with one typical use of
$\mathsf{RT}^2_2$ over $\mathsf{RCA}_0$, contrary both to our
intuition and to the fact that $\mathsf{RCA}_0 \nvdash^1
\mathsf{RT}^2_2 \rightarrow \mathsf{RT}^2_4$, which follows from
Patey's result~\cite{Patey} that $\mathsf{RT}^2_4 \nleq\sub{c}
\mathsf{RT}^2_2$. In fact, as conjectured by J. Miller~[Hirst and
Mummert, personal communication], this phenomenon is not a
particularity of this and other examples mentioned in~\cite{HirMum},
but is in fact completely general. Indeed, in classical logic, if
$\Gamma \vdash \mathsf{Q} \rightarrow \mathsf{P}$ then we can always
argue in $\Gamma$ as follows: Let $X$ be an instance of
$\mathsf{P}$. Then there are $i$ and $Y$ such that either $i=0$ and
$Y$ is a solution to $X$, or $i=1$ and $Y$ is an instance of
$\mathsf{Q}$ with no solution. If $i=1$ then we get a contradiction
from one use of $\mathsf{Q}$, so $i=0$ and hence $Y$ is a solution to
$X$.

Perhaps more satisfying than the above argument is the following one,
which is directly in the style of the one given in~\cite{HirMum} for
$\mathsf{RT}^2_2$ and $\mathsf{RT}^2_4$.  Let $\Gamma$ be as in
Theorem~\ref{classicalmainthm}, and let $\mathsf{P}$ and $\mathsf{Q}$
be $\Pi^1_2$-problems such that $\Gamma \vdash \mathsf{Q} \rightarrow
\mathsf{P}$. Let $\Theta_n$ and $\Delta_n$ be as in
Theorem~\ref{classicalauxthm}. By that theorem, there is an $n$ such
that $\Gamma \vdash \Delta_n$. The following proof can be carried out
in $\Gamma$.

Let $X_0$ be an instance of $\mathsf{P}$. For each $k=0,\ldots,n$ in
turn, proceed as follows. Given $X_0,\ldots,X_k$,
$e_0,\ldots,e_{k-1}$, and $Y_0,\ldots,Y_{k-1}$, let $e_k$ and $Y_k$ be
such that
\[
\forall X_{k+1}\, \exists e_{k+1}, Y_{k+1}\, \cdots \forall X_n\,
\exists e_n, Y_n\, \Theta_n(e_0, \ldots, e_n,X_0, \ldots, X_n,Y_0,
\ldots, Y_n).
\]
If $Y_k$ is a solution to $X_0$ then let $Y=Y_k$ and let
$i=0$. Otherwise, $Y_k$ is an instance of $\mathsf{Q}$. Either that
instance has a solution or not. If it does not then let $Y=Y_k$ and
let $i=1$. If it does, then let $X_{k+1}$ be such a solution.

By the definition of $\Theta_n$, we must eventually define $Y$, $i$,
and $j$. If $i=1$ then $Y$ is an instance of $\mathsf{Q}$ with no
solution. But with one application of $\mathsf{Q}$, we can obtain a
solution to $Y$, so we must have $i=0$, and hence $Y$ is a solution to
$X_0$.
\end{rem}

\section{Computable winning strategies}
\label{compsec}

We now turn to the notion of generalized Weihrauch reducibility for
games over possibly nonstandard models. Let $\Gamma$ be a set of
$L_2$-formulas consistent with $\Delta^0_1$-comprehensiaon that proves
the existence of a universal $\Sigma^0_1$ formula. Let $\mathsf{P}$
and $\mathsf{Q}$ be $\Pi^1_2$-problems. A \emph{computable strategy}
for Player 2 in $G^\Gamma(\mathsf{Q} \rightarrow \mathsf{P})$ or
$\widehat{G}^\Gamma(\mathsf{Q} \rightarrow \mathsf{P})$ consists of
Player 2 playing according to the formula $e=\Phi_k(n-1)$ (in the
sense of Definition~\ref{accordingdef}) for some $k \in \omega$.

\begin{rem}
To be precise, in the above definition we also need to have a
mechanism to distinguish computably when Player 2 has played a winning
move. Formally, we can simply slightly alter our games so that a move
by Player 2 is either $\{n+1 : n \in Y\}$ where $Y$ is a
$\mathsf{Q}$-instance or $\{0\} \cup \{n+1 : n \in Y\}$ where $Y$ is a
solution to Player 1's first move $X_0$.
\end{rem}

Combining Theorem~\ref{defthm} and Proposition~\ref{defprop} gives us
the following.

\begin{prop}
\label{gwgameprop}
Let $\Gamma$ be a consistent extension of $\Delta^0_1$-comprehension
that proves the existence of a universal $\Sigma^0_1$ formula, and let
$\mathsf{P}$ and $\mathsf{Q}$ be $\Pi^1_2$-problems. Then the
following are equivalent.
\begin{enumerate}[\rm (1)]

\item Player 2 has a computable winning strategy for
$G^\Gamma(\mathsf{Q} \rightarrow \mathsf{P})$.

\item Player 2 has a computable winning strategy for
$\widehat{G}^\Gamma(\mathsf{Q} \rightarrow \mathsf{P})$.

\item There is an $n \in \omega$ such that Player 2 has a computable
strategy for $G^\Gamma(\mathsf{Q} \rightarrow \mathsf{P})$ that
ensures victory in at most $n$ many moves.

\item There is an $n \in \omega$ such that Player 2 has a computable
strategy for $\widehat{G}^\Gamma(\mathsf{Q} \rightarrow \mathsf{P})$
that ensures victory in at most $n$ many moves.

\end{enumerate}
Furthermore, $n$ witnesses $\rm (3)$ if{}f it witnesses $\rm (4)$.
\end{prop}

If the conditions in this proposition hold, then we say that
$\mathsf{P}$ is \emph{generalized Weihrauch reducible over $\Gamma$}
to $\mathsf{Q}$, and write $\mathsf{P} \leq\sub{gW}\sup{$\Gamma$}
\mathsf{Q}$. We can of course define an instance-counting version of
this notion, writing $\mathsf{P} \leq\sub{gW}\sup{$\Gamma,n$}
\mathsf{Q}$ if $n+1$ witnesses that item (3) above holds.

As an example of the application of Proposition~\ref{gwgameprop}, we
can obtain an analog of Corollary~\ref{rt2cor}, using the fact that
Hirschfeldt and Jockusch~\cite[Theorem 4.21]{HirJoc} showed that
$\mathsf{RT}^1_{<\infty} \nleq\sub{gW}\sup{$n$} \mathsf{RT}^1_k$ for
all $n$, while Patey~\cite[Theorem 6.0.1]{Patey2} showed that the same
holds for higher exponents. (Notice that Corollary~\ref{rt2cor} itself
works only for exponent $2$, since $\mathsf{RT}^1_{<\infty}$ is
provable in $\mathsf{RCA}_0 + \mathsf{B}\Sigma^0_2$, while
$\mathsf{RT}^n_k$ for $k>1$ and $\mathsf{RT}^n_{<\infty}$ are both
equivalent to $\mathsf{ACA}_0$ over $\mathsf{RCA}_0$ for $n>2$, as
shown by Simpson~\cite{Simpson1ed} using work of
Jockusch~\cite{Jockusch}.)

\begin{cor}
\label{rt2gwcor}
Let $\Gamma$ consist of $\mathsf{RCA}_0$ together with all $\Pi^1_1$
formulas true over the natural numbers. Then $\mathsf{RT}^n_{<\infty}
\nleq\sub{gW}\sup{$\Gamma$} \mathsf{RT}^n_k$ for all $n$ and $k$.
\end{cor}

Kuyper~\cite{Kuyper} studied a notion closely related to this kind of
instance-counting (though he considered only the case where $\Gamma$
is $\mathsf{RCA}_0$). We give a slightly different definition that is
easily seen to be equivalent to his.

\begin{defn}
Let $\mathsf{P}$ and $\mathsf{Q}$ be $\Pi^1_2$-problems. Say that
$\mathsf{P}$ \emph{Weihrauch-reduces to the composition of $n$ many
copies of $\mathsf{Q}$ via $e_0,\ldots,e_n$} if for every
$X_0,\ldots,X_n$,
\begin{multline*}
 \text{ if }X_0\text{ is a } \mathsf{P}\text{-instance then }\\
  \Phi_{e_0}^{X_0} \text{ is a }\mathsf{Q}\text{-instance and if }X_1\text{ is a solution to }\Phi_{e_0}^{X_0}\text{ then }\\\Phi_{e_1}^{X_0 \oplus X_1}\text{ is a }\mathsf{Q}\text{-instance and if }X_2\text{ is a solution to }\Phi_{e_1}^{X_0 \oplus X_1}\text{ then }\\
                \vdots \\
\Phi_{e_{n-1}}^{X_0 \oplus \cdots \oplus X_{n-1}}\text{ is a
}\mathsf{Q}\text{-instance and if }X_n\text{ is a solution to
}\Phi_{e_{n-1}}^{X_0 \oplus \cdots \oplus X_{n-1}}\text{ then }\\
                 \Phi_{e_n}^{X_0 \oplus \cdots \oplus
                   X_n} \text{ is a solution to } X_0.
\end{multline*}
(Note that in the $n=0$ case, this statement becomes
\[\text{ if
}X_0\text{ is a } \mathsf{P}\text{-instance then } \Phi_{e_0}^{X_0}
\text{ is a solution to } X_0\text{.)}\]
\end{defn}

Kuyper considered the situation where there are $n \in \omega$ and
$e_0,\ldots,e_n \in \omega$ such that $\mathsf{RCA}_0$ proves that
$\mathsf{P}$ Weihrauch-reduces to the composition of $n$ many copies
of $\mathsf{Q}$ via $e_0,\ldots,e_n$. For a fixed $n$, it is not
difficult to see that this condition is equivalent to saying that
Player 2 has a a computable winning strategy for
$G^{\Gamma}(\mathsf{Q} \rightarrow \mathsf{P})$ that ensures victory
in exactly $n+1$ many moves, unless it wins earlier by playing an
instance of $\mathsf{Q}$ with no solution. One might think that this
is the same as saying that there is an $n$ such that $\mathsf{P}
\leq\sub{gW}\sup{$\mathsf{RCA}_0,n$} \mathsf{Q}$, and hence by
Proposition~\ref{gwgameprop} to $\mathsf{P}
\leq\sub{gW}\sup{$\mathsf{RCA}_0$} \mathsf{Q}$, but Remark~\ref{01rem}
applies here as well. The example given there shows that it is
possible to have $\mathsf{P} \leq\sub{gW}\sup{$\mathsf{RCA}_0,1$}
\mathsf{Q}$ but not have Kuyper's condition hold. However, Kuyper's
condition is equivalent to $\mathsf{P}
\leq\sub{gW}\sup{$\mathsf{RCA}_0$} \widehat{\mathsf{Q}}$ for the
modified problem $\widehat{\mathsf{Q}}$ defined in that remark, so we
we will express it in this form.

Kuyper~\cite{Kuyper} claimed that his condition is equivalent to a
form of intuitionistically provable
implication. Uftring~\cite{Uftring,Ufpaper} found a counterexample
that shows that Kuyper's argument is flawed. Kuyper
(see~\cite{Uftring,Ufpaper}) proposed fixing his proof by replacing
the condition $\mathsf{P} \leq\sub{gW}\sup{$\mathsf{RCA}_0$}
\widehat{\mathsf{Q}}$ with $\mathsf{P}
\leq\sub{gW}\sup{$\mathsf{RCA}_0+\mathsf{Q}$}
\widehat{\mathsf{Q}}$. Uftring's example shows that it is possible for
Player 2 to have a computable winning strategy for
$G^{\mathsf{RCA}_0+\mathsf{Q}}(\mathsf{Q} \rightarrow \mathsf{P})$ but
not for $G^{\mathsf{RCA}_0}(\mathsf{Q} \rightarrow \mathsf{P})$, in
contrast with the case for general winning strategies in
Proposition~\ref{auxprop}, so we present a version of it now. We will
give another example with the same properties in Section~\ref{exsec}.

\begin{exa}[Uftring~\cite{Uftring,Ufpaper}]
\label{uftex}
The proof of G\"odel's Incompleteness Theorem shows that there is a
primitive recursive predicate $G$ such that $G(n)$ holds for all $n
\in \omega$ but $\mathsf{RCA}_0$ cannot prove $\forall x\, G(x)$. For
$X \neq \emptyset$, write $\mu X$ for the least element of $X$. Let
\[
\mathsf{P} \equiv \forall X\,\exists Y\, \forall x\, G(x)
\] 
and
\[\mathsf{Q} \equiv \forall X\, [X \neq \emptyset \rightarrow \exists Y\,
G(\mu X)]. 
\]
In $G^{\mathsf{RCA}_0+\mathsf{Q}}(\mathsf{Q} \rightarrow \mathsf{P})$,
Player 1's first move $M$ and $X_0$ must be such that $M[X_0]$ is
consistent with $\mathsf{Q}$, so $M[X_0] \vDash \forall x\, G(x)$, and
hence Player 2 can play, say, $\emptyset$ on its first move and
win. In $G^{\mathsf{RCA}_0}(\mathsf{Q} \rightarrow \mathsf{P})$,
however, Player 1 can play an $M \vDash \neg \forall x\, G(x)$,
together with, say, $X_0=\emptyset$. Then this instance of $\mathsf{P}$
has no solution, so the only way for Player 2 to win is eventually to
play an $M$-instance of $\mathsf{Q}$ with no solution, that is, an $X$
such that $M \vDash \neg G(\mu X)$. 

For any model $M$ of $\Sigma^0_1$-$\mathsf{PA}$, we can consider a run in which
Player 1 plays $M$ and then keeps playing $\emptyset$ until Player 2
either declares victory or wins by playing an $M$-instance of
$\mathsf{Q}$ with no solution. (Notice that we can computably
determine if the latter case holds, since the condition $G(\mu X)$ is
computable.) If Player 2 has a computable winning
strategy for $G^{\mathsf{RCA}_0}(\mathsf{Q} \rightarrow \mathsf{P})$,
then there is a computable procedure that, over any model $M$ of
$\Sigma^0_1$-$\mathsf{PA}$, simulates the above run, making Player 2's moves
according to this procedure,
outputting $0$ if Player 2 declares victory, and outputting $\mu X$ if Player
2 plays the $M$-instance $X$ of $\mathsf{Q}$ with no solution. The
output of this procedure is $0$ if{}f $M \vDash \forall x\,
G(x)$. Since this procedure works for any model $M$ of
$\Sigma^0_1$-$\mathsf{PA}$, we have an existential first-order sentence that is
provably equivalent to $\forall x\, G(x)$ over $\mathsf{RCA}_0$, which
is a contradiction, because any existential first-order sentence true
in the standard natural numbers is provable in $\mathsf{RCA}_0$.
\end{exa}

For some $\Pi^1_2$-problems $\mathsf{Q}$, on the other hand, there is
no difference between $G^{\mathsf{RCA}_0+\mathsf{Q}}(\mathsf{Q}
\rightarrow \mathsf{P})$ and $G^{\mathsf{RCA}_0}(\mathsf{Q}
\rightarrow \mathsf{P})$ because every countable model of
$\mathsf{RCA}_0$ can be extended to a countable model of
$\mathsf{RCA}_0+\mathsf{Q}$ with the same first-order part, and hence
the notion of consistency used in Definition~\ref{redgamedef} is the
same for $\mathsf{RCA}_0$ and $\mathsf{RCA}_0+\mathsf{Q}$. (Showing
that this is the case for a given $\mathsf{Q}$ is typically done to
show that $\mathsf{Q}$ is $\Pi^1_1$-conservative over
$\mathsf{RCA}_0$.) Examples include $\textsf{WKL}$, as shown by
Harrington (see~\cite[Theorem IX.2.1]{Simpson}), $\mathsf{COH}$, as
shown by Cholak, Jockusch, and Slaman~\cite{CJS}, and $\mathsf{AMT}$,
as shown by Hirschfeldt, Shore, and Slaman~\cite{HSS}.

As highlighted by the work of Kuyper and Uftring, the connections with
intuitionistic provability are rather subtle, and we believe that
generalized Weihrauch reducibility over possibly nonstandard models
can be useful in clarifying them. However, as the methods and issues
are rather different from the ones in this paper, we leave this work
to a future one.

\section{Single-instance reductions}
\label{singlesec}

As noted in Remark~\ref{01rem}, $\mathsf{P} \leq\sub{c} \mathsf{Q}$
if{}f Player 2 has a strategy for $G(\mathsf{Q} \rightarrow
\mathsf{P})$ that ensures victory in exactly two moves. Similarly,
$\mathsf{P} \leq\sub{W} \mathsf{Q}$ if{}f Player 2 has a computable
strategy for $G(\mathsf{Q} \rightarrow \mathsf{P})$ that ensures
victory in exactly two moves. We can define the analogous notions for
games over possibly nonstandard models. Let us explicitly define these
analogs for computable and Weihrauch reducibilities, and then look at
several examples involving them. Although we will not work with them
in this paper, we also define the analogs of several related notions
of computability-theoretic reduction between $\Pi^1_2$-problems.

\begin{defn}
Let $\Gamma$ be a set of $L_2$-formulas consistent with
$\Delta^0_1$-{\penalty0\hskip0pt\relax}comprehension that proves the
existence of a universal $\Sigma^0_1$ formula and let $\mathsf{P}$ and
$\mathsf{Q}$ be $\Pi^1_2$-problems.
\begin{enumerate}

\item We say that $\mathsf{P}$ is \emph{computably reducible over
$\Gamma$} to $\mathsf{Q}$, and write $\mathsf{P}
\leq\sub{c}\sup{$\Gamma$} \mathsf{Q}$, if for every model $(M,S)$ of
$\Gamma$ with $M$ countable and $S$ closed under
$\Delta^0_1$-comprehension, and every $M$-instance $X$ of $\mathsf{P}$
in $S$, there is an $M$-instance $\widehat{X}$ of $\mathsf{Q}$ in
$M[X]$ such that for every solution $\widehat{Y}$ to $\widehat{X}$ in
$S$, there is a solution to $X$ in $M[X,\widehat{Y}]$.

\item  We say that $\mathsf{P}$ is \emph{Weihrauch reducible over
$\Gamma$} to $\mathsf{Q}$, and write $\mathsf{P}
\leq\sub{W}\sup{$\Gamma$} \mathsf{Q}$, if there are $e,i \in \omega$
such that for every model $(M,S)$ of $\Gamma$ with $M$ countable and
$S$ closed under $\Delta^0_1$-comprehension, and every $M$-instance 
$X$ of $\mathsf{P}$ in $S$, the set $\widehat{X}=\Phi_e^X$ is an
$M$-instance of $\mathsf{Q}$, and for every solution $\widehat{Y}$
to $\widehat{X}$ in $S$, the set $\Phi_i^{X \oplus \widehat{Y}}$ is
a solution to $X$.

\item We say that $\mathsf{P}$ is \emph{strongly computably reducible
over $\Gamma$} to $\mathsf{Q}$, and write $\mathsf{P}
\leq\sub{sc}\sup{$\Gamma$} \mathsf{Q}$, if for every model $(M,S)$ of
$\Gamma$ with $M$ countable and $S$ closed under
$\Delta^0_1$-comprehension, and every $M$-instance $X$ of $\mathsf{P}$
in $S$, there is an $M$-instance $\widehat{X}$ of $\mathsf{Q}$ in
$M[X]$ such that for every solution $\widehat{Y}$ to $\widehat{X}$ in
$S$, there is a solution to $X$ in $M[\widehat{Y}]$.

\item  We say that $\mathsf{P}$ is \emph{strongly Weihrauch reducible
over $\Gamma$} to $\mathsf{Q}$, and write $\mathsf{P}
\leq\sub{sW}\sup{$\Gamma$} \mathsf{Q}$, if there are $e,i \in \omega$
such that for every model $(M,S)$ of $\Gamma$ with $M$ countable and
$S$ closed under $\Delta^0_1$-comprehension, and every $M$-instance
$X$ of $\mathsf{P}$ in $S$, the set $\widehat{X}=\Phi_e^X$ is an
$M$-instance of $\mathsf{Q}$, and for every solution $\widehat{Y}$ to
$\widehat{X}$ in $S$, the set $\Phi_i^{\widehat{Y}}$ is a solution to
$X$.

\item We say that $\mathsf{P}$ is \emph{omnisciently computably
reducible over $\Gamma$} to $\mathsf{Q}$, and write $\mathsf{P}
\leq\sub{oc}\sup{$\Gamma$} \mathsf{Q}$, if for every model $(M,S)$ of
$\Gamma$ with $M$ countable and $S$ closed under
$\Delta^0_1$-comprehension, and every $M$-instance $X$ of $\mathsf{P}$
in $S$, there is an $M$-instance $\widehat{X}$ of $\mathsf{Q}$ in $S$
such that for every solution $\widehat{Y}$  to $\widehat{X}$ in $S$,
there is a solution to $X$ in $M[X,\widehat{Y}]$.

\item We say that $\mathsf{P}$ is \emph{omnisciently Weihrauch
reducible over $\Gamma$} to $\mathsf{Q}$, and write $\mathsf{P}
\leq\sub{oW}\sup{$\Gamma$} \mathsf{Q}$, if there is an $i \in \omega$
such that for every model $(M,S)$ of $\Gamma$ with $M$ countable and
$S$ closed under $\Delta^0_1$-comprehension, and every $M$-instance
$X$ of $\mathsf{P}$ in $S$, there is an $M$-instance $\widehat{X}$ of
$\mathsf{Q}$ in $S$ such that for every solution $\widehat{Y}$ to
$\widehat{X}$ in $S$, the set $\Phi_i^{X \oplus \widehat{Y}}$ is a
solution to $X$.

\item We say that $\mathsf{P}$ is \emph{strongly omnisciently
computably reducible over $\Gamma$} to $\mathsf{Q}$, and write
$\mathsf{P} \leq\sub{soc}\sup{$\Gamma$} \mathsf{Q}$, if for every
model $(M,S)$ of $\Gamma$ with $M$ countable and $S$ closed under
$\Delta^0_1$-comprehension, and every $M$-instance $X$ of $\mathsf{P}$ 
in $S$, there is an $M$-instance $\widehat{X}$ of $\mathsf{Q}$ in
$S$ such that for every solution $\widehat{Y}$ to $\widehat{X}$ in
$S$, there is a solution to $X$ in $M[\widehat{Y}]$.

\item We say that $\mathsf{P}$ is \emph{strongly omnisciently
Weihrauch reducible over $\Gamma$} to $\mathsf{Q}$, and write
$\mathsf{P} \leq\sub{soW}\sup{$\Gamma$} \mathsf{Q}$, if there is an $i
\in \omega$ such that for every model $(M,S)$ of $\Gamma$ with $M$
countable and $S$ closed under $\Delta^0_1$-comprehension, and every
$M$-instance $X$ of $\mathsf{P}$ in $S$, there is an $M$-instance
$\widehat{X}$ of $\mathsf{Q}$ in $S$ such that for every solution
$\widehat{Y}$ to $\widehat{X}$ in $S$, the set
$\Phi_i^{\widehat{Y}}$ is a solution to $X$.

\end{enumerate}
\end{defn}

\begin{rem}
In light of comments made above, it might be more natural to consider
versions corresponding to games in which Player 2 can always win in
one or two moves, rather than exactly two moves (even if in natural
cases, there will be no difference). Rather than introduce more
terminology and notation, however, that can be done simply by
replacing $\mathsf{Q}$ with the problem $\widehat{\mathsf{Q}}$ from
Remark~\ref{01rem} in the above definitions.
\end{rem}

The study of Weihrauch reducibility in this extended setting seems
particularly promising, given the extensive theory that has been
developed for Weihrauch reducibility over the standard natural
numbers. In particular, there are several operators on the Weihrauch
degrees whose analogs in this setting should be of interest. One
example is the finite parallelization: For a problem $\mathsf{P}$, the
problem $\mathsf{P}^*$ is the one whose instances consist of finitely
many instances $X_0,\ldots,X_k$ of $\mathsf{P}$, with a solution
consisting of one solution to each $X_i$. Clearly, $\mathsf{P}^*
\leq\sub{gW} \mathsf{P}$ for any $\Pi^1_2$-problem $\mathsf{P}$, but
this fact does not hold in our setting, because given an instance
$X_0,\ldots,X_k$ of $\mathsf{P}^*$, the obvious reduction strategy for
Player 2 takes $k+1$ many moves, and $k$ might be nonstandard. The
following example will be relevant in the next section.

\begin{exa}
\label{boundex}
Pauly, Fouch\'e, and Davie~\cite{PFD} defined $\mathsf{Bound}$ as
follows: An instance is an enumeration of a bounded set $F$, and a
solution is a bound on the elements of $F$. An instance of
$\mathsf{Bound}^*$ is then a simultaneous enumeration of a finite
family $F_0,\ldots,F_k$ of bounded sets, and a solution to this
instance consists of a bound for each $F_i$, or, equivalently, a bound
$b$ on $\bigcup_{i \leq k} F_k$. (This is basically the principle
$\mathsf{FUF}$ studied by Frittaion and Marcone \cite{FM}.) It is easy
to see that $\mathsf{Bound}$ and $\mathsf{Bound}^*$ are
Weihrauch-equivalent, but that is no longer the case for
Weihrauch-equivalence (or even provable equivalence) over
$\mathsf{RCA}_0$, since as statements in second-order arithmetic,
$\mathsf{Bound}$ is trivially true, while $\mathsf{Bound}^*$ is a way
to state $\mathsf{B}\Pi^0_1$, and hence is equivalent to
$\mathsf{B}\Sigma^0_2$ over $\mathsf{RCA}_0$, as we
further discuss in the following section. Thus $\mathsf{RCA}_0 \nvdash
\mathsf{Bound} \rightarrow \mathsf{Bound}^*$, and hence
$\mathsf{Bound}^* \nleq\sub{gW}\sup{$\mathsf{RCA}_0$} \mathsf{Bound}$.
\end{exa}

It is not clear what the correct generalization of the ${}^\diamond$
operator of Neumann and Pauly~\cite{NeuPau} to this setting
is. However, one would expect that it would still have the property
that $\mathsf{P}^*$ is reducible to $\mathsf{P}^\diamond$, and hence,
by the above example, that it would no longer be equivalent to
gW-reducibility.

On the other hand, it is clear that, as for standard Weihrauch
reducibility, if $\mathsf{P} \leq\sub{W}\sup{$\mathsf{RCA}_0$}
\mathsf{Q}$ then $\mathsf{P}^* \leq\sub{W}\sup{$\mathsf{RCA}_0$}
\mathsf{Q}^*$. It is also not difficult to see that, more generally,
if $\mathsf{P} \leq\sub{gW}\sup{$\mathsf{RCA}_0,n$} \mathsf{Q}$ then
$\mathsf{P}^* \leq\sub{gW}\sup{$\mathsf{RCA}_0,n$} \mathsf{Q}^*$. Thus,
by Proposition~\ref{gwgameprop}, if $\mathsf{P}
\leq\sub{gW}\sup{$\mathsf{RCA}_0$} \mathsf{Q}$ then $\mathsf{P}^*
\leq\sub{gW}\sup{$\mathsf{RCA}_0$} \mathsf{Q}^*$. (The same holds for
other appropriate systems in place of $\mathsf{RCA}_0$, of course.)

An important point here is that while the principles we consider in
reverse mathematics are typically true---in the sense that they hold
in $(\omega,\mathcal P(\omega))$, or equivalently for
$\Pi^1_2$-problems, that every instance (over the standard natural
numbers) has at least one solution---many of them have nontrivial
first-order parts. For example, if $\mathsf{B}\Sigma^0_2$ fails in
$M$, then $M$ cannot be the first-order part of a model of
$\mathsf{RCA}_0 + \mathsf{RT}^1_{<\infty}$ (or of $\mathsf{RCA}_0 +
\mathsf{RT}^n_k$ for any $n,k \geq 2$). Furthermore, for any such $M$
there is an instance of $\mathsf{RT}^1_{<\infty}$ (i.e., a $k \in |M|$
together with a function $c : |M| \rightarrow \{j \in |M| : j <^M
k\}$) with no solutions. The same is true of $\mathsf{Bound}^*$, to
give another example.

We want to use notions such as Weihrauch reducibility over
$\mathsf{RCA}_0$ and other systems to study these kinds of principles
(as we will do in the next two sections), so it is important that our
definitions above do not assume that every instance of a problem has a
solution. This fact is particularly worth noting for Weihrauch
reducibility, because we usually think of (classical) Weihrauch
reducibility between $\Pi^1_2$-problems as a special case of the
general notion from computable analysis, which is defined using
partial multifunctions between represented spaces. (See for
instance Brattka, Gherardi, and Pauly~\cite{BGP} or Brattka and
Pauly~\cite{BP}.) This point is a bit subtle, and was missed, e.g., in
the paper Dorais, Dzhafarov, Hirst, Mileti, and Shafer \cite{DDHMS},
where a proof is given in Corollaries A.3 and A.4 establishing a
correspondence between $\Pi^1_2$ principles on the one hand and
certain classes of partial multifunctions on the other. Indeed, the
proof there works only if the $\Pi^1_2$ principles in question are
assumed to be true, which is not explicitly mentioned.

There is more than one way to formalize the notion of a partial
multifunction between spaces $X$ and $Y$. One is to say that it is
simply a relation $R \subseteq X \times Y$. Then the domain of the
multifunction is $\{x \in X : \exists y\, (x,y) \in R\}$. Another is
to say that it is a (possibly partial) function from $X$ to the power
set of $Y$. In this case, the domain of the multifunction can include
elements that are mapped to no values at all. The first formalization
is the one normally used in the definition of Weihrauch reducibility
in computable analysis, which is convenient in particular because of
the need to use choice functions in working with represented
spaces. And indeed, a true $\Pi^1_2$-problem $\mathsf{P}$ corresponds
to the partial multifunction $F : \mathord{\subseteq} 2^\omega
\rightrightarrows 2^\omega$ in this sense whose domain is the set of
instances of $\mathsf{P}$, and which maps any such instance $X$ to the
solutions to $X$.

This correspondence breaks down for a $\Pi^1_2$-problem that has
instances with no solutions, however, unless we move to the second
formalization of the notion of multifunction, or allow a multifunction
to consist of a relation $R \subseteq X \times Y$ together with a set
$D$ such that $\{x \in X : \exists y\, (x,y) \in R\} \subseteq D
\subseteq X$, where $D$ represents the domain of the function.  This
distinction operates even at the level of the Weihrauch degrees
(equivalence classes under Weihrauch reducibility), because a problem
in which some instance has no solutions can never be Weihrauch
reducible to one in which every instance has a solution, and if
$\mathsf{P}$ has a computable instance with no solutions, then every
problem is Weihrauch reducible to $\mathsf{P}$. As discussed
in~\cite{BGP}, and in more detail in~\cite{BP}, this top degree is
usually added to the lattice of Weihrauch degrees as a formal object.

The distinction between the two approaches is also relevant to the
notion of \emph{extended Weihrauch reducibility} investigated by Bauer
\cite{B} (see also \cite{B-talk}), following work by Bauer and
Yoshimura \cite{BY-CCA,BY-CCC}. The focus in that work is on comparing
universally quantified statements in the setting of constructive
mathematics, using a notion called \emph{instance reducibility}, which
can also be understood as an extension of the Weihrauch degrees that
in particular allows for ``questions that do not have an answer'' but
that are still ``valid'' for the purposes of considering whether or
not they are reducible to other questions (Bauer \cite{Bauer-slides}).

\section{Limit-homogeneous sets}
\label{exsec}

In this section and the next, we give some examples of comparisons of
$\Pi^1_2$-problems using W- and gW-reducibility over possibly
nonstandard models, focusing on versions of $\mathsf{B}\Sigma^0_2$. A
natural way to think of $\mathsf{B}\Pi^0_1$ as a $\Pi^1_2$-problem is
to identify a $\Pi^0_1$ formula $\phi(i,k)$ with a simultaneous
enumeration of the sets $\{m : \forall k<m\, \neg\phi(i,k)\}$ for
$i<n$. Then a $b$ as in the definition of $\mathsf{B}\Pi^0_1$ is the
same as a common bound for these sets. Thus we arrive at
$\mathsf{Bound}^*$, as defined in Example~\ref{boundex}.

Recall also the $\Pi^1_2$-problems $\mathsf{SRT}^2_2$ and
$\mathsf{D}^2_2$ from Definition~\ref{rtdef}. Clearly,
$\mathsf{SRT}^2_2$ implies $\mathsf{D}^2_2$. Cholak, Jockusch, and
Slaman~\cite{CJS} claimed that the converse implication also holds
over $\mathsf{RCA}_0$, but their proof actually required
$\mathsf{B}\Sigma^0_2$. Chong, Lempp, and Yang~\cite{CLY} closed this
gap by showing that $\mathsf{D}^2_2$ implies $\mathsf{B}\Sigma^0_2$
over $\mathsf{RCA}_0$.

The argument in~\cite{CJS} also shows that $\mathsf{SRT}^2_2
\leq\sub{c} \mathsf{D}^2_2$. Dzhafarov~\cite{Dzhafarov2} and Brattka
and Rakotoniaina~\cite{BraRak} showed that $\mathsf{SRT}^2_2
\nleq\sub{W} \mathsf{D}^2_2$. Hirschfeldt and Jockusch~\cite{HirJoc}
noted that $\mathsf{SRT}^2_2 \leq\sub{gW}^2 \mathsf{D}^2_2$,
however. To consider this reduction in more detail, we define the
following $\Pi^1_2$-problem.

\begin{defn} 
$\mathsf{LH}$: If $c : [\mathbb N]^2 \rightarrow 2$ is such that
$\lim_y c(x,y)=1$ for all $x$, then $c$ has an infinite homogeneous
set.
\end{defn}

\noindent This problem is a convenient way to state the principle that
for every $2$-coloring of pairs, every infinite limit-homogeneous set
has an infinite homogeneous subset.

From the reverse-mathematical perspective, $\mathsf{LH}$ is equivalent
to $\mathsf{B}\Sigma^0_2$.

\begin{prop}
\label{lhprop}
$\mathsf{RCA}_0 \vdash \mathsf{LH} \leftrightarrow
\mathsf{B}\Sigma^0_2$.
\end{prop}

\begin{proof}
First, assume $\mathsf{B}\Sigma^0_2$. Fix an instance $c$ of
$\mathsf{LH}$. Let $S$ be the set of all tuples
$(x_0,\ldots,x_{n-1},y)$ such that $x_0 < \cdots < x_{n-1} < y$ and
$c(x_m,y) = 1$ for all $m < n$. We claim that for all $x_0 < \cdots <
x_{n-1}$, there is a $y$ such that $(x_0,\ldots,x_{n-1},y) \in S$. For
each $m < n$ there is a $b_m > x_{n-1}$ such that $c(x_m,y) = 1$ for
all $y > b_m$. By $\mathsf{B}\Sigma^0_2$ (or really
$\mathsf{B}\Pi^0_1$), there is a $b > x_{n-1}$ such that $c(x_m,y) =
1$ for all $m < n$ and $y > b$. Then $(x_0,\ldots,x_{n-1},b+1) \in S$,
which proves our claim. Now we can define a homogeneous set $H$ for
$c$ by primitive recursion: Let $h_0=0$, let $h_{n+1}$ be the least
$y$ such that $(h_0,\ldots,h_n,y) \in S$, and let
$H=\{h_0,h_1,\ldots\}$.

Now assume $\mathsf{LH}$. We prove $\mathsf{RT}^1_{<\infty}$. Assume
for a contradiction that $d : \mathbb N \rightarrow k$ has no infinite
homogeneous set. Then for each $i < k$ there is a $b$ such that $d(x)
\neq i$ for all $x > b$. Define $c : [\mathbb N]^2 \rightarrow 2$ by
letting $c(x,y)=0$ if $d(x)=d(y)$ and letting $c(x,y)=1$
otherwise. Then $\lim_y c(x,y) = 1$ for all $x$, so by $\mathsf{LH}$,
$c$ has an infinite homogeneous set $H$. Let $x_0 < \cdots < x_k \in
H$. Then for all $m < n \leq k$, we have that $c(x_m,x_n) = 1$ and
hence $d(x_m) \neq d(x_n)$. But then $\{d(x_0),\ldots,d(x_k)\}$ has
cardinality $k+1$, which is impossible.
\end{proof}

However, the first part of the above proof shows that $\mathsf{LH}$ is
computability-theoretically trivial, and indeed uniformly
computably true, so that $\mathsf{LH} \leq\sub{gW}^0 \mathsf{P}$ for
any $\mathsf{P}$, or equivalently $\mathsf{LH} \leq\sub{W}
\mathsf{1}$, where $\mathsf{1}$ is the identity problem for which an
instance is any $X$ and the only solution to this instance is $X$
itself. We can obtain $\mathsf{SRT}^2_2$ from
$\mathsf{D}^2_2$ as follows: Given a stable coloring $c : \mathbb{N}
\rightarrow 2$, use $\mathsf{D}^2_2$ to obtain a limit-homogeneous set
$L$. Now an application of $\mathsf{RT}^1_2$ (which is
Weihrauch-reducible to $\mathsf{D}^2_2$) yields an $i$ such that
$\lim_{y \in L} c(x,y)=i$ for all $x \in L$. We can think of $c$
restricted to $L$ as a coloring of $\mathbb N$ by identifying the
$n$th element of $L$ with $n$. If $i=0$, we can also replace $c$ by
the coloring whose value at $(x,y)$ is $1-c(x,y)$. We can then apply
$\mathsf{LH}$ to obtain an infinite homogeneous set for $c$. Since
$\mathsf{LH}$ is Weihrauch-trivial, this procedure shows that
$\mathsf{SRT}^2_2 \leq\sub{gW}^2 \mathsf{D}^2_2$. (Since the use of
$\mathsf{RT}^1_2$ is computably trivial, it also shows that
$\mathsf{SRT}^2_2 \leq\sub{c} \mathsf{D}^2_2$, as mentioned above.)

Over nonstandard models, however, things are different. In the
presence of $\mathsf{B}\Sigma^0_2$, the first part of the proof of
Proposition~\ref{lhprop} shows that $\mathsf{LH}$ is still
Weihrauch-trivial, i.e., $\mathsf{LH}
\leq\sub{W}\sup{$\mathsf{RCA}_0+\mathsf{B}\Sigma^0_2$} \mathsf{1}$,
and hence $\mathsf{SRT}^2_2
\leq\sub{gW}\sup{$\mathsf{RCA}_0+\mathsf{B}\Sigma^0_2,\, 2$}
\mathsf{D}^2_2$. Of course, if $\mathsf{P}$ does not imply
$\mathsf{B}\Sigma^0_2$ over $\mathsf{RCA}_0$, then we cannot have
$\mathsf{LH} \leq\sub{gW}\sup{$\mathsf{RCA}_0$} \mathsf{P}$. But what
if we take $\mathsf{P}$ to be some form of $\mathsf{B}\Sigma^0_2$? A
natural choice is $\mathsf{Bound}^*$, as it is essentially the form of
$\mathsf{B}\Pi^0_1$ used in the first part of the proof of
Proposition~\ref{lhprop}.

We will show that $\mathsf{LH} \nleq\sub{gW}\sup{$\mathsf{RCA}_0$}
\mathsf{Bound}^*$, but we can actually obtain a stronger result by
considering the contrapositive form of $\mathsf{B}\Pi^0_1$: Given a
simultaneous enumeration of sets $F_0,\ldots,F_{n-1}$ with no common
bound, there is an $i<n$ such that $F_i$ is infinite. Given such an
enumeration, we can define an $n$-coloring $c$ of $\mathbb N$ as
follows: for each $m$, wait until a number greater than $m$ is
enumerated into some $F_i$, then give $m$ the color $i$. From an
infinite homogeneous set for $c$, we can obtain an $i<n$ such that
$F_i$ is infinite. Conversely, given an $n$-coloring $c$ of $\mathbb
N$, the sets $F_i=\{m : c(m)=i\}$ for $i<n$ have no common bound, and
from an $i<n$ such that $F_i$ is infinite, we can obtain an infinite
homogeneous set for $c$. Both of these processes can be carried out
over $\mathsf{RCA}_0$, so up to Weihrauch equivalence over
$\mathsf{RCA}_0$, the contrapositive form of $\mathsf{B}\Pi^0_1$ is
$\mathsf{RT}^1_{<\infty}$, in the form in which it is usually stated
as a $\Pi^1_2$-problem, in which an instance consists of a
$k$-coloring of $\mathbb N$ together with the number $k$.

\begin{rem}
The above argument (which we heard from Pauly [personal
communication]) also gives a simple proof of Hirst's result
from~\cite{Hirst} (see also~\cite[Theorem 6.81]{Hirschfeldt}) that
$\mathsf{B}\Sigma^0_2$ and $\mathsf{RT}^1_{<\infty}$ are equivalent
over $\mathsf{RCA}_0$.
\end{rem}

There is a stronger form of $\mathsf{RT}^1_{<\infty}$, which we will
call $\mathsf{stRT}^1_{<\infty}$, in which the number of colors is not
part of the instance. That is, an instance consists of a function
$\mathbb N \rightarrow \mathbb N$ with bounded range (and a solution
is still an infinite homogeneous set). As shown by Brattka and
Rakotoniaina~\cite{BraRak}, and also noted by Hirschfeldt and
Jockusch~\cite{HirJoc}, $\mathsf{RT}^1_{<\infty} <\sub{W}
\mathsf{stRT}^1_{<\infty}$.  In this section, we show that
$\mathsf{LH} \nleq\sub{gW}\sup{$\mathsf{RCA}_0$}
\mathsf{stRT}^1_{<\infty}$.  We will show in
Proposition~\ref{boundandstrt} that $\mathsf{Bound}^*
\leq\sub{gW}\sup{$\mathsf{RCA}_0$} \mathsf{stRT}^1_{<\infty}$, so this
result implies that $\mathsf{LH} \nleq\sub{gW}\sup{$\mathsf{RCA}_0$}
\mathsf{Bound}^*$, but we also give a direct proof of the latter fact,
which uses the same technique but is simpler.

Both proofs will use the following notion of forcing.

\begin{defn}
\label{notion}
Let $N$ be an $L_1$-structure. We define a notion of forcing $P_N$ as
follows. (If $N$ is the standard natural numbers then we denote this
notion by $P_\omega$.) Write $[m]^2$ for the set of $(x,y) \in
[|N|]^2$ such that $x,y <^N m$. A condition is an $N$-finite function
of the form $p : [m]^2 \rightarrow 2$ for some $m \in |N|$. Say that a
condition $q$ extends such a $p$ if $q$ extends $p$ as a function and
$q(i,j)=1$ for all $i <^N m$ and $j \geq^N m$ on which it is
defined. Define the notion of $c : [|N|]^2 \rightarrow 2$ extending
$p$ in the same way. (Notice that if for every $m \in |N|$ there is a
condition $p : [m]^2 \rightarrow 2$ such that $c$ extends $p$, then
$c$ is an $N$-instance of $\mathsf{LH}$.)
\end{defn}

We will also use the following fact. (A \emph{$1$-elementary
extension} of a structure $N$ is an extension of $N$ that satisfies
exactly the same existential sentences with parameters from $N$.)

\begin{lem}
\label{extlem}
There is a $1$-elementary extension $M$ of the standard natural
numbers such that for the collection $S$ of all subsets of $|M|$ that
are $\Delta^0_1$-definable over $M$,
\begin{enumerate}

\item $(M,S)$ is a model of $\mathsf{RCA}_0$ and 

\item for any condition $p$ for the notion of forcing $P_M$, there is
an $M$-instance of $\mathsf{LH}$ in $S$ that extends $p$ (in the
sense of Definition~\ref{notion}) and has no solution in $S$.

\end{enumerate}
\end{lem}

\begin{proof}
Let $N$ be any nonstandard elementary extension of the standard
natural numbers, and let $a \in N$ be a nonstandard element. Then in
particular $N \vDash \mathsf{I}\Sigma^0_2$, and so
\[
M = \{x \in N : x \text{ is } \Sigma^0_2\text{-definable in } (N,a)\}
\]
is a model of $\mathsf{I}\Sigma^0_1 + \neg \mathsf{B}\Sigma^0_2$ which
is a $1$-elementary (in fact, $2$-elementary) substructure of
$N$. (See H\'{a}jek and Pudlak \cite[Theorem IV.1.33]{HP-1993} or
Kossak \cite[p.~223]{Kossak-1990}.) Thus, $M$ is a $1$-elementary
extension of the standard model, and for $S$ as in the statement,
$(M,S)$ is a model of $\RCA_0$. Since $\mathsf{B}\Sigma^0_2$ fails in
$M$, it follows by Proposition \ref{lhprop} that $\mathsf{LH}$ fails
in $(M,S)$. Fix an instance $c : [M]^2 \to 2$ of $\mathsf{LH}$ in $S$
with no solution in $S$. Then given a condition $p : [m]^2 \to 2$ for
$P_M$, we can define $d : [M]^2 \to 2$ by
\[
d(x,y) = \begin{cases}
 p(x,y) & \text{if } x,y < m	,\\
 1 & \text{if } x < m \text{ and } y \geq m,\\
 c(x,y) & \text{otherwise.}
 \end{cases}
\]
Clearly, $d$ is in $S$ and is an instance of $\mathsf{LH}$ that
extends $p$. But if $H$ is any solution to $d$ then $\{ x \in H : x
\geq m\}$ is a solution to $c$, so $d$ cannot have any solution in
$S$.
\end{proof}

\begin{prop}
\label{boundprop}
$\mathsf{LH} \nleq\sub{gW}\sup{$\mathsf{RCA}_0$} \mathsf{Bound}^*$.
\end{prop}

\begin{proof}
Assume for a contradiction that $\mathsf{LH}
\leq\sub{gW}\sup{$\mathsf{RCA}_0$} \mathsf{Bound}^*$. By
Proposition~\ref{gwgameprop}, there is an $n \in \omega$ such that
Player 2 has a computable strategy for
$\widehat{G}^{\mathsf{RCA}_0}(\mathsf{Bound}^* \rightarrow \mathsf{LH})$ that
ensures victory in at most $n$ many moves. Fix such a strategy.

For a condition $p : [j]^2 \rightarrow 2$ for the notion of forcing
$P_\omega$, we can consider what happens when our fixed strategy for
Player 2 is applied to a run in which Player 1 plays $(\omega,\mathcal
P(\omega))$ and $p$ as a partial first move. Unless the strategy
declares victory on its first move, it must play part of an instance
of $\mathsf{Bound}^*$, which is just a simultaneous enumeration of a
finite family of sets. We may assume by the usual convention on uses
that no number greater than $j$ is enumerated. Let $b^p_0$ be the
least bound on the set of all numbers enumerated in this way. Now, if
Player 1 plays $b^p_0$, then unless our strategy declares victory on
its second move, it again must play part of an instance of
$\mathsf{Bound}^*$, yielding an analogous bound $b^p_1$. Continuing in
this way, we obtain numbers $b^p_0,b^p_1,\ldots,b^p_k$ for some
$k<n$. Let $b^p_i=0$ for $k < i < n$.

For $i<n$ and $m \in \omega$, let $D_{i,m}$ be the set of conditions
$p$ such that $b^p_i \geq m$. If some $D_{0,m}$ is not dense then let
$m_0$ be the least such $m$. In this case, there is a condition $p_0
\in D_{0,m_0-1}$ with no extension in $D_{0,m_0}$. Notice that
$b^q_0=m_0-1$ for all extensions $q$ of $p_0$. Now, if some $D_{1,m}$ is
not dense below $p_0$ then let $m_1$ be the least such $m$. In this
case, there is an extension of $p_0$ in $D_{1,m_1-1}$ with no
extension in $D_{1,m_1}$. Proceeding in this way, we obtain a
condition $p$ such that either $m_i$ is defined for every $i<n$, or
there is a $k<n$ such that $m_i$ is defined for all $i<k$ and every
$D_{k,m}$ is dense below $p$. In either case, $b^q_i=m_i-1$ for all
extensions $q$ of $p$ and all $i$ such that $m_i$ is defined.

We claim that the latter case cannot hold. Suppose otherwise. Let $c$
be an instance of $\mathsf{LH}$ that extends $p$ and meets every
$D_{k,m}$ (i.e., every $D_{k,m}$ contains a $q$ such that $c$ extends
$q$). Then Player 1 can play $(\omega,\mathcal P(\omega))$ and $c$ on
its first move, and if Player 2 follows our fixed strategy, then the
moves $m_0,m_1,\ldots,m_{k-1}$ will be legal for Player 1 (as
otherwise some finite portion of $c$ is a condition $q$ extending $p$
with $b^q_i > m_i$ for some $i<k$). But then Player 2's $(k+1)$st move
is not an instance of $\mathsf{Bound}^*$.

Thus each $m_i$ for $i<n$ is defined, and we have the following for
our fixed condition $p$:
\begin{equation}
\label{mdef}
\forall q \, \forall i<n\, [\textrm{if $q$ extends
$p$ then } b^q_i=m_i-1].
\end{equation}
Now let $M$ and $S$ be as in Lemma~\ref{extlem}. Then $p$ is also a
condition for $P_M$, so there is an $M$-instance $d$ of $\mathsf{LH}$
in $S$ that extends $p$ and has no solution in $S$. But it is easy to
check that (\ref{mdef}) is a $\Pi^0_1$ statement, so since $M$ is a
$1$-elementary extension of the standard natural numbers, it also
holds over $M$. So Player $1$ can play $(M,S)$ and $d$ on its first
move, and if Player 2 follows our fixed strategy, then the moves
$m_0,m_1,\ldots,m_{n-1}$ will be legal for Player 1 (as otherwise some
finite portion of $d$ is a condition $q$ extending $p$ with $b^q_i >^M
m_i$ for some $i<n$). But then Player 2 has not won the game by the
$n$th move (since the only way for Player 2 to win this run of the
game is to play an $M$-instance of $\mathsf{Bound}^*$ with no
solution), contrary to assumption.
\end{proof}

Thus $\mathsf{LH}$ and $\mathsf{Bound}^*$ constitute a natural example of
the phenomenon witnessed by Uftring's Example~\ref{uftex}.

We can also interpret the fact that $\mathsf{LH}
\leq\sub{W}\sup{$\mathsf{RCA}_0+\mathsf{Bound}^*$} \mathsf{1}$ but
$\mathsf{LH} \nleq\sub{gW}\sup{$\mathsf{RCA}_0$} \mathsf{Bound}^*$ as
saying that the use of $\mathsf{Bound}^*$ in the first part of the
proof of Proposition~\ref{lhprop} is ``purely proof-theoretic''. It
neither requires a further ``computability-theoretic application'' of
$\mathsf{Bound}^*$ nor can be replaced by one or more such
applications (in the uniform setting). Uncovering this kind of
information seems to be a promising aspect of this approach to
calibrating the logical strength of $\Pi^1_2$-problems.

Proposition~\ref{boundprop} does not show that $\mathsf{SRT}^2_2
\nleq\sub{gW}\sup{$\mathsf{RCA}_0$} \mathsf{D}^2_2$, but it suggests
that this might well be the case, which would provide an even more
natural version of Example~\ref{uftex}, and show that the proof of
$\mathsf{SRT}^2_2$ from $\mathsf{D}^2_2$ necessarily makes both
computability-theoretic and further proof-theoretic use of
$\mathsf{D}^2_2$. Indeed, it even seems possible that $\mathsf{LH}
\nleq\sub{gW}\sup{$\mathsf{RCA}_0$} \mathsf{D}^2_2$.

\begin{q}
Is $\mathsf{SRT}^2_2 \leq\sub{gW}\sup{$\mathsf{RCA}_0$}
\mathsf{D}^2_2$? Is $\mathsf{LH} \leq\sub{gW}\sup{$\mathsf{RCA}_0$}
\mathsf{D}^2_2$?
\end{q}

We now strengthen Proposition~\ref{boundprop} as described above.

\begin{prop}
\label{strtprop}
$\mathsf{LH} \nleq\sub{gW}\sup{$\mathsf{RCA}_0$}
\mathsf{stRT}^1_{<\infty}$.
\end{prop}

\begin{proof}
Assume for a contradiction that $\mathsf{LH}
\leq\sub{gW}\sup{$\mathsf{RCA}_0$} \mathsf{stRT}^1_{<\infty}$. By
Proposition~\ref{gwgameprop}, there is an $n \in \omega$ such that
Player 2 has a computable strategy for
$\widehat{G}^{\mathsf{RCA}_0}(\mathsf{stRT}^1_{<\infty} \rightarrow
\mathsf{LH})$ ensuring victory in at most $n$ many moves. There is
then also a strategy that ensures victory in exactly $n$ many moves,
since Player 2 can extend the length of any game by playing computable
($\Delta^0_1$-definable) instances of $\mathsf{stRT}^1_{<\infty}$ on
all its moves from some point on. Fix such a strategy, and for
notational convenience, assume $n > 1$.

We begin as in the previous proof by considering games over the
standard natural numbers. Note that if Player 2 plays according to
its strategy and does not declare victory on some move, then it has to
play an instance of $\mathsf{stRT}^1_{<\infty}$ only provided all of
Player 1's moves so far have been legal. However, since every set can
be viewed as a coloring $\omega \to \omega$ (not necessarily with
bounded range), we can always assume that Player 2 plays such a
coloring. This coloring may be partial, however, in which case by
usual use conventions we can assume it is defined on a finite initial
segment of $\omega$.

Fix a condition $p$ for the notion of forcing $P_\omega$. For each
$\alpha \in \omega^{\leq n-2}$, we define a coloring $f^p_\alpha$ of a
finite initial segment of $\omega$. Having done so, we let
$H^p_{\alpha^\frown v}$ for each $v \in \omega$ be the set of all $x$
such that $f^p_\alpha(x) = v$. We start with $\alpha$ equal to
$\lambda$, the empty string. As in the proof of Proposition
\ref{boundprop}, suppose Player 1 plays $(\omega,\mathcal{P}(\omega))$
and $p$ as a partial first move. Since $n > 1$, the strategy for
Player 2 makes it play a coloring of a finite initial segment of
$\omega$ as its partial first move. Let $f^p_\lambda$ be this
coloring. Now, suppose $f^p_\alpha$ has been defined for some $\alpha$
with $|\alpha| < n - 2$, and fix $v \in \omega$. Suppose Player 1
plays $(\omega,\mathcal{P}(\omega))$ and $p$ as a partial first move,
and for $0 < k \leq |\alpha| + 1$, plays $H^p_{(\alpha^\frown v)
\upharpoonright k}$ as a partial $(k+1)$st move, with Player 2
playing according to its fixed strategy. Since $n > |\alpha|+2$, the
strategy for Player 2 makes it again play a coloring of an initial
segment as its partial $(|\alpha|+2)$nd move. Let $f^p_{\alpha^\frown
v}$ be this coloring.

We now define a finitely branching subtree $T$ of $\omega^{\leq
n-2}$, and for each $\alpha \in T$, a condition $p_\alpha$, such
that the following properties hold:
\begin{enumerate}

\item For all $\alpha,\beta \in T$, if $\beta$
length-lexicographically precedes $\alpha$ then $p_\alpha$ extends
$p_\beta$.

\item For every $\alpha^\frown v \in T$ and for every $m \in \omega$,
the set of conditions $p$ with $f^p_\alpha(x) = v$ for some $x \geq
m$ is dense below $p_\alpha$.

\item For every $\alpha \in T$ and every $v$ such that $\alpha^\frown
v \in \omega^{n-1} \setminus T$, if $f^p_\alpha(x) = v$ for some
condition $p$ extending $p_\alpha$ and some $x$, then $x \in \dom
f^{p_\alpha}_\alpha$.

\end{enumerate}

We put strings $\alpha$ into $T$ and define $p_\alpha$
simultaneously. Initially, put $\lambda \in T$ and let $p_\lambda$ be
the empty condition. Notice that properties 1--3 hold vacuously at
this point.

Next, assume we are at a point in the definition of $T$ at which
properties 1--3 hold, and consider the length-lexicographically least
$\alpha \in T$ with $|\alpha| < n-2$ such that we have not yet put
$\alpha^\frown{} v$ into $T$ for any $v$. Let $\beta \in T$ be
length-lexicographically largest such that $p_\beta$ has been
defined. Let $c : \omega \to \omega$ extend $p$ and be sufficiently
generic for the forcing notion $P_\omega$. If Player 1 plays
$(\omega,\mathcal{P}(\omega))$ and $c$ on its first move, then the
strategy for Player 2 makes it play an instance $f_0$ of
$\mathsf{stRT}^1_{<\infty}$ in response. By property 2 and the genericity
of $c$, the set of $x$ such that $f_0(x) = \alpha(0)$ is infinite, so
$H_0 = \{x : f_0(x) = \alpha(0)\}$ is a legal second move for Player
1. Then the strategy for Player 2 makes it play another instance $f_1$
of $\mathsf{stRT}^1_{<\infty}$ on its second move, and the set $H_1 =
\{x : f_1(x) = \alpha(1)\}$ will be infinite and hence a legal third
move for Player 1. Since $n > |\alpha|+2$, if we continue in this way
we analogously define $f_0,\ldots,f_{|\alpha|}$ and
$H_0,\ldots,H_{|\alpha|-1}$, with $f_k$ played by Player 2 on its
$(k+1)$st move for all $k \leq |\alpha|$, and $H_k$ played by Player 1
on its $(k+2)$nd move for all $k < |\alpha|$. Since the strategy for
Player 2 is computable and hence continuous, it is easy to see by
induction that if $q$ is any condition extended by $c$ then
$f^q_{\alpha \upharpoonright k}$ is an initial segment of $f_k$, and
$H^p_{\alpha \upharpoonright (k + 1)}$ is an initial segment of
$H_k$. Now, as $f_{|\alpha|}$ is an instance of
$\mathsf{stRT}^1_{<\infty}$, there must be a condition $q_0$ extended
by $c$ and a $b \in \omega$ such that $f^r_\alpha(x) < b$ for all $x$
and all $r$ extending $q_0$.

We now decide for which $v < b$ to add $\alpha^\frown{} v$ to $T$ and
define $p_{\alpha^\frown{}v}$. Fix $v$, and suppose we have already
decided this for all $w < v$. For notational convenience, assume we
have also defined an auxiliary condition $q_v$ extending $q_0$. If
there is a condition $r$ extending $q_v$ such that for every $m \in
\omega$, every extension of $r$ has a further extension $s$ such that
$f^s_\alpha(x) = v$ for some $x \geq m$, then let $\alpha^\frown{}v
\in T$ and let $p_{\alpha^\frown{}v} = q_{v+1} = r$. Otherwise, there
is an extension $r$ of $q_v$ such that for every extension $s$ of $r$,
if $f^s_\alpha(x) = v$ for some $x$ then $x$ is in the domain of
$f^r_\alpha$, and we let $q_{v+1} = r$ and let $\alpha^\frown{}v
\notin T$. It is readily seen that this process adds
$\alpha^\frown{}v$ to $T$ for at least one $v$, and for only finitely
many $v$, and that properties 1, 2, and 3 are preserved.

Let $p^* = p_\beta$ for the length-lexicographically largest $\beta
\in T$. Let $M$ be as given by Lemma \ref{extlem}, and let $S$ be the
set of subsets of $|M|$ that are $\Delta^0_1$-definable over
$M$. Every condition for $P_\omega$ is also a condition for $P_M$. So
let $c$ be an instance of $\mathsf{LH}$ in $S$ that extends $p^*$ and
has no solution in $S$. For every node $\alpha \in T$, let $G_\alpha$
be the following run of a game. Player 1 plays $(M,S)$ and $c$ as its
first move, and Player 2 plays according to its strategy. On its
$(k+1)$st move for $0 < k \leq |\alpha|$, Player 1 always plays the
set of all $x \in M$ that are colored $\alpha(k-1)$ by the coloring
played by Player 2 on its previous move (assuming it played a total
coloring and not just a partial one). We claim that there is an
$\alpha \in T$ of length $n-2$ such that Player 1's moves in
$G_\alpha$ are all legal. We argue by induction (along the standard
number $n-1$) that for each $k < n-1$ there is such an $\alpha \in T$
of length $k$. Suppose that for some $\alpha \in T$ of length $k-1$,
Player 1's moves in $G_\alpha$ are all legal. Then on its
$(|\alpha|+1)$st move in $G_\alpha$, Player 2 plays an instance $f$ of
$\mathsf{stRT}^1_{<\infty}$. Now, property 3 in the definition of $T$
is a $\Pi^0_1$ statement of arithmetic, so since $M$ is a
$1$-elementary extension of $\omega$, it must also hold in $M$. Thus,
all the $v \in M$ such that $f^{-1}(v)$ is unbounded in $M$ must be
among those for which $\alpha^\frown v \in T$. Since there are only
standardly many such $v$, there must be at least one for which
$f^{-1}(v)$ really is unbounded in $M$, so Player 1's moves in
$G_{\alpha^\frown v}$ will all be legal. This establishes the
claim. To complete the proof, fix such an $\alpha$ of length
$n-2$. All sets played by Player 1 are clearly in $S$, so when Player
2 declares victory on its $n$th (i.e., $(|\alpha|+2)$nd) move in
$G_\alpha$ it must play a solution to $c$ in $S$. But there is no such
solution by hypothesis, which is a contradiction.
\end{proof}

\section{Versions of $\mathbf{\Pi^0_1}$-bounding}
\label{exsec2}

In this section we fill out the picture of implications between
versions of $\mathsf{B}\Pi^0_1$ and related principles.

As with $\mathsf{RT}^1_{<\infty}$, we can define a strong form
$\mathsf{stBound}^*$ of $\mathsf{Bound}^*$ by having the number of
sets not be part of the instance. A convenient way to express this
problem is to say that an instance is an enumeration of a subset $X$
of $\mathbb N \times \mathbb N$ such that $\{n : \exists k\, (n,k) \in
X\}$ is bounded, and for each $n$, so is the set $\{k : (n,k) \in X\}$;
and a solution is a bound on $\{k : \exists n\, (n,k) \in
X\}$. It is easy to see that $\mathsf{stBound^*}
\equiv\sub{W} \mathsf{Bound}$, but we will see that this equivalence
no longer holds in our setting.

Another problem worth mentioning in this connection is
$\mathsf{C}_{\mathbb N}$, for which an instance is an enumeration of
the complement of a nonempty set $X$, and a solution is an element of
$X$. The finite parallelization $\mathsf{C}_{\mathbb N}^*$ is yet
another equivalent of $\mathsf{B}\Sigma^0_2$: In one direction, we can
enumerate the sets $\{m : \forall k<m\, \neg\phi(i,k)\}$ for a given
$\Pi^0_1$ formula $\phi(i,k)$, and from a tuple containing an element
of the complement of each of these sets, obtain a common bound on the
sets. In the other direction, given simultaneous enumerations of the
complements of the nonempty sets $F_0,\ldots,F_n$, by
$\textsf{B}\Pi^0_1$, there is a $b$ such that each $F_i$ has an
element less than $b$. Now bounded $\Pi^0_1$-comprehension, which
holds in $\mathsf{RCA}_0$, gives us the set of all tuples
$(a_0,\ldots,a_j)$ with $j \leq n$ and $a_i \in F_i$ for all $i \leq
j$, and set induction shows that there must be such a tuple with
$j=n$.

It is easy to see that $\mathsf{C}_{\mathbb N} \equiv\sub{W}
\mathsf{C}_{\mathbb N}^*$, and Pauly, Fouch\'e, and Davie~\cite{PFD}
showed that $\mathsf{Bound} \equiv\sub{W} \mathsf{C}_{\mathbb N}$,
using the Weihrauch equivalence between $\mathsf{C}_{\mathbb N}$ and
its restriction $\mathsf{UC}_{\mathbb N}$ to enumerations of
complements of singleton sets, which was proved by Brattka, de Brecht,
and Pauly~\cite{BBP}. Brattka and Rakotoniaina~\cite{BraRak} showed
that $\mathsf{C}_{\mathbb N} \mid\sub{W} \mathsf{RT}^1_{<\infty}$ and
$\mathsf{C}_{\mathbb N} <\sub{W} \mathsf{stRT}^1_{<\infty}$. Indeed,
it is even the case that $\mathsf{RT}^1_2 \nleq\sub{W}
\mathsf{C}_{\mathbb N}$; we will prove a stronger version of this fact
below. It is also worth noting that $\mathsf{RT}^1_2 <\sub{W}
\mathsf{RT}^1_3 <\sub{W} \cdots$, as shown by Brattka and
Rakotoniaina~\cite{BraRak} and Hirschfeldt and
Jockusch~\cite{HirJoc}. Thus we have the following picture for
Weihrauch reducibility:

\bigskip

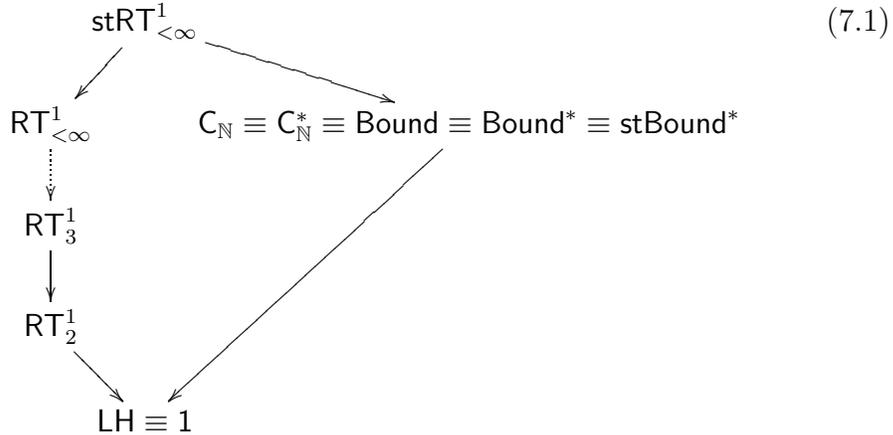
\begin{equation}
\label{figone}
\xymatrix@C-30pt@R-5pt{
 & {\mathsf{stRT}^1_{<\infty}}\ar[dl]\ar[dr] \\
{\mathsf{RT}^1_{<\infty}} \ar@{.>}[d] & & {\mathsf{C}_{\mathbb N} \equiv
\mathsf{C}_{\mathbb N}^* \equiv \mathsf{Bound} \equiv \mathsf{Bound}^*
\equiv \mathsf{stBound}^*}\ar[dddl] \\
{\mathsf{RT}^1_3}\ar[d]\\
{\mathsf{RT}^1_2}\ar[dr]\\
 & {\mathsf{LH} \equiv \mathsf{1}}
}
\end{equation}

\bigskip

Hirschfeldt and Jockusch~\cite[Proposition 4.7]{HirJoc} showed that
$\mathsf{RT}^1_{<\infty} \leq\sub{gW} \mathsf{RT}^1_2$, but their
proof in fact shows that $\mathsf{stRT}^1_{<\infty} \leq\sub{gW}
\mathsf{RT}^1_2$. On the other hand, we have the following.

\begin{prop}
$\mathsf{RT}^1_2 \nleq\sub{gW} \mathsf{C}_{\mathbb N}$.
\end{prop}

\begin{proof}
Suppose that $\mathsf{RT}^1_2 \leq\sub{gW} \mathsf{C}_{\mathbb N}$
via a computable strategy $P$ for Player 2. As Player 1, we can begin
to build a coloring $c$ by coloring numbers in order, initially giving
each number the color $0$, and simulate the action of $P$. We can
assume that, even when provided with inputs that do not correspond to
a run of $G(\mathsf{C}_{\mathbb N} \rightarrow \mathsf{RT}^1_2)$, if
$P$ does not declare victory at a given move, then it outputs an
enumeration of the complement of some set, though in that case the set
might be empty.

Let $A_i$ be the set whose complement is being enumerated by $P$ as
its $(i+1)$st move (if $P$ has not declared victory at or before that
move). We guess at each stage that the least number $k_i$ currently in
$A_i$ is a solution to the corresponding instance of
$\mathsf{C}_{\mathbb N}$ and play that as our $(i+2)$nd move in the
simulation. If we ever find that $k_i$ is not in $A_i$, we restart the
simulation (but do not change $c$ on the numbers at which we have
already defined it). For the least such $i$, say that $i$ causes the
simulation to restart. If the current simulation is not restarted,
then eventually $P$ must declare victory at some move, and declare
some number $m$ to be in the set it outputs at that move. We then
start to give our numbers the color $1-c(m)$. If we were to do this
forever, then $m$ could not be part of a solution to $c$, so our
current simulation cannot be a true run of the game, and hence
eventually some $i$ must cause it to restart.

Thus the simulation is restarted infinitely often. There are now two
cases.

If there is a least $i$ that causes the simulation to restart
infinitely often, then, by induction, $k_0,\ldots,k_{i-1}$ have final
values, and if we play $c$ on our first move, and then play these
values in turn, we produce a run of our game in which $P$'s $(i+1)$st
move is an enumeration of $\mathbb N$, and hence is not an instance of
$\mathsf{C}_{\mathbb N}$, which is a contradiction.

Otherwise, again by induction, all $k_i$'s have final values, and if
we play $c$ on our first move, and then play these values in turn, we
produce a run of our game in which $P$ never declares victory, which
is again a contradiction.
\end{proof}

So for gW-reducibility, we have the following simpler picture:

\bigskip

\begin{equation}
\label{figtwo}
\xymatrix@C-30pt@R-10pt{
{\mathsf{stRT}^1_{<\infty} \equiv \mathsf{RT}^1_{<\infty}} \equiv
\mathsf{RT}^1_2 \ar[d] \\
{\mathsf{C}_{\mathbb N} \equiv
\mathsf{C}_{\mathbb N}^* \equiv \mathsf{Bound} \equiv \mathsf{Bound}^*
\equiv \mathsf{stBound}^*} \ar[d] \\
{\mathsf{LH} \equiv \mathsf{1}}
}
\end{equation}

\bigskip

It is easy to check that all the Weihrauch reductions in
Diagram~(\ref{figone}) still work over $\mathsf{RCA}_0 +
\mathsf{B}\Sigma^0_2$, so that diagram also reflects the relationships
between these principles with respect to
$\leq\sub{W}\sup{$\mathsf{RCA}_0 + \mathsf{B}\Sigma^0_2$}$ (or
$\leq\sub{W}\sup{$\Gamma$}$ for any extension $\Gamma$ of
$\mathsf{RCA}_0 + \mathsf{B}\Sigma^0_2$ by formulas true over the
natural numbers). Diagram~(\ref{figtwo}), however, does change if we
work over $\mathsf{RCA}_0 + \mathsf{B}\Sigma^0_2$. We still have the
equivalence between $\mathsf{RT}^1_j$ and $\mathsf{RT}^1_k$ for $j,k
\geq 2$ (which holds even over $\mathsf{RCA}_0$, with the usual
proof), but Corollary~\ref{rt2gwcor} shows that
$\mathsf{RT}^1_{<\infty}
\nleq\sub{gW}\sup{$\mathsf{RCA}_0+\mathsf{B}\Sigma^0_2$}
\mathsf{RT}^1_k$ for all $n$ and $k$. Similarly, we have the
following.

\begin{prop}
$\mathsf{C}_{\mathbb N} \nleq\sub{gW}\sup{$n$}
\mathsf{RT}^1_{<\infty}$ for all $n$, so if we let $\Gamma$ consist of
$\mathsf{RCA}_0$ together with all $\Pi^1_1$ formulas true over the
natural numbers then  $\mathsf{C}_{\mathbb N}
\nleq\sub{gW}\sup{$\Gamma$} \mathsf{RT}^1_{<\infty}$.
\end{prop}

\begin{proof}
Suppose that $\mathsf{C}_{\mathbb N} \leq\sub{gW}\sup{$n$}
\mathsf{RT}^1_{<\infty}$ via a computable strategy $P$ for Player
2. We can assume that, even when provided with inputs that do not
correspond to a run of $G(\mathsf{RT}^1_{<\infty} \rightarrow
\mathsf{C}_{\mathbb N})$, if $P$ does not declare victory at a given
move, then its output at that move, if nonempty, is a number $k$
together with a possibly partial $c : \mathbb N \rightarrow k$.

For a possibly partial $c : \mathbb N \rightarrow k$, let $H_c =
\{c^{-1}(0),\ldots,c^{-1}(k-1)\}$. Note that if $c$ is total then at
least one element of $H_c$ is a solution to $c$ as an instance of
$\mathsf{RT}^1_{<\infty}$. We can start building an instance $E$ of
$\mathsf{C}_{\mathbb N}$ by initially not enumerating any numbers, and
running simulations of possible runs of $G(\mathsf{RT}^1_{<\infty} \rightarrow \mathsf{C}_{\mathbb N})$ beginning with $E$, where each
time $P$ plays some $c$, we play a simulation for each possible move
for Player 1 in $H_c$. (Notice that $c$ might not actually be an
instance of $\mathsf{RT}^1_{<\infty}$ because this simulation might not
correspond to an actual run of the game, but $H_c$ is still
finite. This is the reason we could not work with
$\mathsf{stRT}^1_{<\infty}$ here, because in that case $P$ would be
able to play functions with unbounded range during simulations that do
not correspond to actual runs.)

Whenever in any of these simulations $P$ declares victory at or before
the $(n+1)$st move with a purported solution $m$, we enumerate $m$
into $E$. Since each $H_c$ is finite, and we consider only finitely
many $c$'s during this construction, we enumerate only finitely many
numbers into $E$, and this strategy ensures that there is a run of
$G(\mathsf{RT}^1_{<\infty} \rightarrow \mathsf{C}_{\mathbb N})$
beginning with $E$ in which either $P$ does not declare victory by
its $(n+1)$st move, or it does so with a purported solution $m$ that
is enumerated into $E$, and hence is not in fact a solution to $E$. In
either case we have a contradiction.

The second part of the proposition now follows from
Proposition~\ref{gwgameprop}.
\end{proof}

Thus we have the following picture for gW-reducibility over
$\mathsf{RCA}_0 + \mathsf{B}\Sigma^0_2$  (or over any extension of
$\mathsf{RCA}_0 + \mathsf{B}\Sigma^0_2$ by $\Pi^1_1$ formulas true
over the natural numbers):

\bigskip

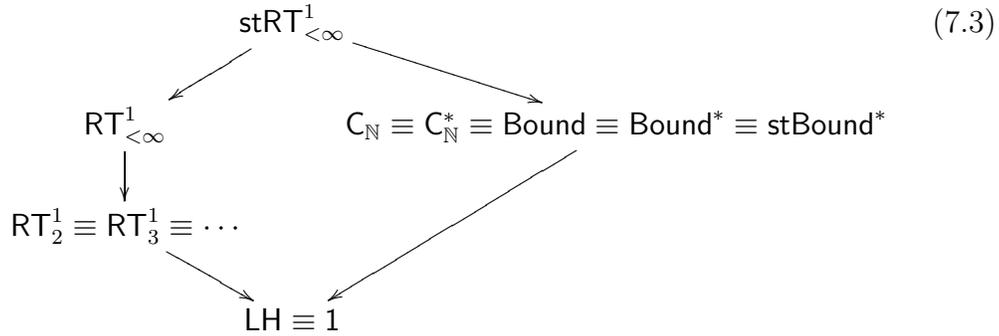
\begin{equation}
\label{figthree}
\xymatrix@C-30pt@R-5pt{
 & {\mathsf{stRT}^1_{<\infty}}\ar[dl]\ar[dr] \\
{\mathsf{RT}^1_{<\infty}} \ar[d] & & {\mathsf{C}_{\mathbb N} \equiv
\mathsf{C}_{\mathbb N}^* \equiv \mathsf{Bound} \equiv \mathsf{Bound}^*
\equiv \mathsf{stBound}^*} \ar[ddl] \\
{\mathsf{RT}^1_2 \equiv \mathsf{RT}^1_3 \equiv \cdots} \ar[dr] \\
 & {\mathsf{LH} \equiv \mathsf{1}}
}
\end{equation}

\bigskip

When working over $\mathsf{RCA}_0$, things change even further. We do
still have $\mathsf{Bound} \equiv\sub{W}\sup{$\mathsf{RCA}_0$}
\mathsf{C}_{\mathbb N}$, $\mathsf{Bound}^*
\equiv\sub{W}\sup{$\mathsf{RCA}_0$} \mathsf{C}_{\mathbb N}^*$, and
$\mathsf{C}_{\mathbb N} \leq\sub{W}\sup{$\mathsf{RCA}_0$}
\mathsf{stRT}^1_{<\infty}$, with essentially the same proofs. The only
parts that require a bit of care are $\mathsf{C}_{\mathbb N}
\leq\sub{W}\sup{$\mathsf{RCA}_0$} \mathsf{Bound}$ and
$\mathsf{C}_{\mathbb N}^* \leq\sub{W}\sup{$\mathsf{RCA}_0$}
\mathsf{Bound}^*$. We prove the latter, as the former is similar but
simpler. We argue in $\mathsf{RCA}_0$. Given an enumeration of the
complements of nonempty sets $A_0,\ldots,A_n$, constituting an
instance of $\mathsf{C}_{\mathbb N}^*$, we define enumerations of sets
$F_0,\ldots,F_n$ by putting $s$ into $F_i$ whenever the least element
$m^i_s$ of $A_i$ at stage $s$ of the enumeration of its complement
leaves $A_i$ at that stage. If $F_i$ were unbounded, then so would be
the set of numbers $m^i_s$, since the map taking $F_i$ to this set is
injective and computable. But then $A_i$ would be empty. So each $F_i$
is bounded, and hence our enumeration of $F_0,\ldots,F_n$ is an
instance of $\mathsf{Bound}^*$. If $s$ is a solution to this instance
then for each $i \leq n$, the least element of $A_i$ at stage $s$ must
be in $A_i$, so from $s$ we obtain a solution to our instance of
$\mathsf{C}_{\mathbb N}^*$.

However, every instance of $\mathsf{C}_{\mathbb N}$ and
$\mathsf{Bound}$ in every model of $\mathsf{RCA}_0$ has a solution,
while this is not the case for $\mathsf{C}_{\mathbb N}^*$ and
$\mathsf{Bound}^*$, which are equivalent to $\mathsf{B}\Sigma^0_2$
over $\mathsf{RCA}_0$ as statements of second-order arithmetic. So
$\mathsf{C}_{\mathbb N}$ is strictly below $\mathsf{C}_\mathbb{N}^*$
under both $\leq\sub{W}\sup{$\mathsf{RCA}_0$}$ and
$\leq\sub{gW}\sup{$\mathsf{RCA}_0$}$, and similarly for
$\mathsf{Bound}$ and $\mathsf{Bound}^*$.

We also no longer have a Weihrauch-reduction of $\mathsf{stBound}^*$
to $\mathsf{Bound}^*$, but do have one in two steps, because an
instance of $\mathsf{Bound}^*$ (or even $\mathsf{Bound}$) can be used
to determine the number of sets being enumerated in an instance of
$\mathsf{stBound}^*$, allowing us to solve that instance with a second
application of $\mathsf{Bound}^*$.

\begin{prop}
$\mathsf{stBound}^* \leq\sub{gW}\sup{$\mathsf{RCA}_0$,2}
\mathsf{Bound}^*$ but $\mathsf{stBound}^*
\nleq\sub{W}\sup{$\mathsf{RCA}_0$} \mathsf{Bound}^*$.
\end{prop}

\begin{proof}
Given an instance $X$ of $\mathsf{stBound}^*$, we can first build an
instance of $\mathsf{Bound}$ by enumerating $n$ whenever $X$
enumerates $(n,k)$ for some $k$. Given a solution $b$ to this
instance, we can build an instance of $\mathsf{Bound}^*$ consisting of
enumerations of sets $F_0,\ldots,F_{b-1}$ by enumerating $k$ into
$F_n$ whenever $X$ enumerates $(n,k)$. A solution to this instance is
also a solution to $X$.

For the second part, suppose that $\mathsf{stBound}^*
\leq\sub{W}\sup{$\mathsf{RCA}_0$} \mathsf{Bound}^*$ via
$\Phi_e$ and $\Phi_i$. An enumeration $E$ of $\emptyset$ is an
instance of $\mathsf{stBound}^*$, so $\Phi_e^E$ must be an instance of
$\mathsf{Bound}^*$. This instance has a fixed number of sets
$k$, which must be the same standard natural number no matter what
model of $\mathsf{RCA}_0$ we are working in, because the convergent
computation over the standard natural numbers still exists in any such
model. Now let $(M,S)$ be a model of $\mathsf{RCA}_0$ that contains an
$M$-instance $D$ of $\mathsf{stBound}^*$ with no solution. We can
delay $D$ to define a new $M$-instance $\widehat{D}$ of
$\mathsf{stBound}^*$ that enumerates the same set as $D$ but agrees
with $E$ up to the use of the part of the computation of $\Phi_e^E$
that fixes the number of sets at $k$. Then $\widehat{D}$ has no
solution, but $\Phi_e^{\widehat{D}}$ is an instance of
$\mathsf{Bound}^*$ with a standard number of sets, and hence
must have a solution $b$. But then $\Phi_i$ should be able to
compute a solution to $\widehat{D}$ from $\widehat{D}$ and $b$,
which is a contradiction.
\end{proof}

We can make the first part of this proposition a bit more precise by
using the compositional product from the theory of Weihrauch
reducibility: $\mathsf{stBound}^* \leq\sub{W}\sup{$\mathsf{RCA}_0$}
\mathsf{Bound}^* \star \mathsf{Bound}$.

The second part of the proposition easily generalizes to establish the
following useful principle (which we state for $\mathsf{RCA}_0$ but of
course applies to other systems as well).

\begin{prop}
Let $\mathsf{P}$ and $\mathsf{Q}$ be $\Pi^1_2$-problems such that
\begin{enumerate}

\item $\mathsf{P}$ has an $\omega$-instance $X$ such that for any
finite initial segment $\sigma$ of $X$, there is a model $(M,S)$ of
$\mathsf{RCA}_0$ and an $M$-instance $Y$ of $\mathsf{P}$ in $S$ that
extends $\sigma$ and has no solution in $S$; and

\item every instance $X$ of $\mathsf{Q}$ includes a parameter $k_X \in
\mathbb N$ such that for every model $(M,S)$ of $\mathsf{RCA}_0$ and
every $M$-instance $X$ of $\mathsf{Q}$ in $S$, if $k_X$ is a standard
natural number, then $X$ has a solution in $S$.

\end{enumerate}
Then $\mathsf{P} \nleq\sub{W}\sup{$\mathsf{RCA}_0$} \mathsf{Q}$.
\end{prop}

As an example of the application of this principle, we have the
following.

\begin{cor}
$\mathsf{LH} \nleq\sub{W}\sup{$\mathsf{RCA}_0$}
\mathsf{RT}^1_{<\infty}$.
\end{cor}

We also have the following other example of a W-reducibility that
becomes a gW-reducibility in two steps when generalized to models of
$\mathsf{RCA}_0$.

\begin{prop}
\label{boundandstrt}
$\mathsf{stBound}^* \leq\sub{gW}\sup{$\mathsf{RCA}_0$,2}
\mathsf{stRT}^1_{<\infty}$ but $\mathsf{Bound}^*
\nleq\sub{W}\sup{$\mathsf{RCA}_0$} \mathsf{stRT}^1_{<\infty}$.
\end{prop}

\begin{proof}
For the first part, we argue in $\mathsf{RCA}_0$. Given an instance
$X$ of $\mathsf{stBound}^*$, let $E_{i,n}$ be the set of $k$ such that
$(i,k)$ has been enumerated into $X$ by stage $n$, and let $i_n$ be
the least $i$ that maximizes $\max E_{i,n}$ (which exists because the
function taking $i$ to $\max E_{i,n}$ is computable). We first produce
an instance of $\mathsf{stRT}^1_{<\infty}$ by giving $n$ the color
$i_n$. Given a solution $H$ to this instance, let $i$ be the color of
the elements of $H$. Now apply $\mathsf{Bound}$ (which is W-reducible
over $\mathsf{RCA}_0$ to $\mathsf{stRT}^1_{<\infty}$) to obtain a
bound $b$ on $\{k : (i,k) \in X\}$. This bound must be a solution to
$X$, because if $(j,k) \in X$ for some $j$ and $k>b$, then once
$(j,k)$ is enumerated into $X$ at some stage $m$, we cannot have
$i_n=i$ for $n \geq m$.

Now suppose that $\mathsf{Bound}^* \leq\sub{W}\sup{$\mathsf{RCA}_0$}
\mathsf{stRT}^1_{<\infty}$ via $\Phi_e$ and $\Phi_i$. We work over a
model $M$ of $\Sigma^0_1$-$\mathsf{PA}$ that satisfies
$\Sigma^0_2$-bounding but not $\Sigma^0_3$-bounding. Then there is a
$\Delta^0_2$ $M$-instance $c : |M| \rightarrow k$ of
$\mathsf{RT}^1_{<\infty}$ with no solution. Say that sets
$F_0,\ldots,F_{k-1}$ are acceptable if $c(n)=i$ for every $i<k$ and $n
\in F_i$. Notice that in this case, each $F_i$ is bounded, so an
enumeration of an acceptable family of sets is an $M$-instance of
$\mathsf{Bound}^*$.

Thinking of $M$-finite enumerations of acceptable families as a notion
of forcing, suppose that for each $j \in M$, the set of such
enumerations $E$ for which some element greater than $j$ is in the
range of $\Phi_e^E$ is dense. Then we can computably build an
enumeration $D$ of an acceptable family such that $\Phi_e^D$ has
unbounded range, and is thus not an instance of
$\mathsf{stRT}^1_{<\infty}$. As this situation cannot happen, there
must be a $j \in M$ and an $M$-finite enumeration $E$ of an
acceptable family such that for every enumeration $D$ of an acceptable
family extending $E$, the range of $\Phi_e^D$ is bounded by $j$.

Now we start building such a $D$ by monitoring $\Phi_i^{D \oplus H_p}$
for each $H_p = \{n : \Phi_e^D(n)=p\}$ with $p <^M j$. Whenever we see
$\Phi_i^{D \oplus H_p}$ return a number $m_p$, we enumerate $m_p+1$
into $F_{c(m_p+1)}$, where $F_0,\ldots,F_{k-1}$ is the family that $D$
is enumerating. The set of $p <^m j$ such that $m_p$ is ever defined
is a bounded $\Sigma^0_1$ set, and the map taking each $p$ in this set
to $m_p$ is computable, so the set of $m_p$'s is $M$-finite. But then
the restriction of $c$ to this set is also $M$-finite, because the
fact that $M$ satisfies $\Sigma^0_2$-bounding implies that the
intersection of a $\Delta^0_2$ set with an $M$-finite set is
$M$-finite. So $D$ is an $M$-finite extension of the $M$-finite
enumeration $E$, and hence is itself $M$-finite, and thus $\Phi_e^D$
is a computable instance of $\mathsf{stRT}^1_{<\infty}$, and hence
must have a solution. But then some $H_p$ with $p <^M j$ must be such
a solution, and hence $\Phi_i^{D \oplus H_p}$ must be a solution to
$D$. But we ensured that this is not the case, so we have a
contradiction.
\end{proof}

The first part of this proof shows more precisely that
$\mathsf{stBound}^* \leq\sub{W}\sup{$\mathsf{RCA}_0$} \mathsf{Bound}
\star \mathsf{stRT}^1_{<\infty}$ and that $\mathsf{Bound}^*
\leq\sub{W}\sup{$\mathsf{RCA}_0$} \mathsf{Bound} \star
\mathsf{RT}^1_{<\infty}$.

Combining the results above with Proposition~\ref{strtprop} gives us
the following pictures of the $\leq\sub{W}\sup{$\mathsf{RCA}_0$}$ and
$\leq\sub{gW}\sup{$\mathsf{RCA}_0$}$ cases, respectively.

\bigskip

\begin{equation}
\label{figfour}
\xymatrix@C+30pt{
 & {\mathsf{stRT}^1_{<\infty}} \ar[dl] \ar[dddr]
& & {\mathsf{LH}} \\
{\mathsf{RT}^1_{<\infty}} \ar@{.>}[d] & & {\mathsf{stBound}^*}\ar[d]\\
{\mathsf{RT}^1_3} \ar[d] & & {\mathsf{C}_{\mathbb N}^* \equiv \mathsf{Bound}^*}\ar[d]\\
{\mathsf{RT}^1_2} & & {\mathsf{C}_{\mathbb N} \equiv \mathsf{Bound}}
}
\end{equation}

\bigskip

\bigskip

\begin{equation}
\label{figfive}
\xymatrix@C-10pt{
 & {\mathsf{stRT}^1_{<\infty}} \ar[dl]
\ar[dr] & & {\mathsf{LH}} \\
{\mathsf{RT}^1_{<\infty}} \ar[d]  & & {\mathsf{C}_{\mathbb N}^* \equiv
\mathsf{Bound}^* \equiv \mathsf{stBound}^*}\ar[d]\\
{\mathsf{RT}^1_2 \equiv \mathsf{RT}^1_3 \equiv \cdots} &
& {\mathsf{C}_{\mathbb N} \equiv \mathsf{Bound}}
}
\end{equation}

\end{document}